\documentclass[onecolumn,12pt]{IEEEtran} 
\usepackage{amsfonts}
\usepackage{amsmath, amsthm, amssymb}
\usepackage{color}
\usepackage{graphicx,subfigure}
\usepackage{float}
\usepackage[justification=centering]{caption}

\def\alex#1 {\fbox {\footnote {\ }}\ \footnotetext { From Alex: {\color{blue}#1}}}

\def\qi#1 {\fbox {\footnote {\ }}\ \footnotetext { From Qi: {\color{red}#1}}}
\def\meng#1 {\fbox {\footnote {\ }}\ \footnotetext { From Meng: {\color{blue}#1}}}

\newcommand{\rank}{{\rm Rank}}

\newcommand{\bs}{{\mathbb{S}}}
\newcommand{\gs}{{\rm GS}}
\newcommand{\gl}{{\rm GL}}
\newcommand{\ad}{{\rm ad}}

\newcommand{\F}{\mathbb{F}}

\newcommand{\C}{{\cal C}}

\newtheorem{theorem}{Theorem}
\newtheorem{lemma}[theorem]{Lemma}
\newtheorem{proposition}[theorem]{Proposition}
\newtheorem{corollary}[theorem]{Corollary}

\theoremstyle{definition}
\newtheorem{definition}[theorem]{Definition}
\newtheorem{example}{Example}
\newtheorem{remark}{Remark}

\begin{document}

\title{LDPC Codes Based on the Space of Symmetric Matrices over Finite Fields
\thanks{C. Ma's work was supported by the National Natural Science Foundation of China under Grant No. 11271004 and 11371121. Q. Wang's work was supported by the Shenzhen fundamental research programs under Grant No. JCYJ20150630145302234.
}
}

\author{Meng Zhao, \thanks{M. Zhao is with the Department of Mathematics, Southern University of Science and Technology, Shenzhen, Guangdong, China (email: meng\_zhao08@163.com)}
Changli Ma, \thanks{C. Ma is with the College of Mathematics and Information Science, Hebei Normal University, Shijiazhuang, Hebei, China (email: ma\_changli@hotmail.com)}
and Qi Wang, \IEEEmembership{Member, IEEE}\thanks{Q. Wang is with the Department of Computer Science and Engineering, Southern University of Science and Technology, Shenzhen, Guangdong, China (email: wangqi@sustc.edu.cn).}
}

\maketitle

\begin{abstract}
  In this paper, we present a new method for explicitly constructing regular low-density parity-check (LDPC) codes based on $\bs_{n}(\F_{q})$, the space of $n\times n$ symmetric matrices over $\F_{q}$. Using this method, we obtain two classes of binary LDPC codes, $\C(n,q)$ and $\C^{T}(n,q)$, both of which have grith $8$. Then both the minimum distance and the stopping distance of each class are investigated. It is shown that the minimum distance and the stopping distance of $\C^{T}(n,q)$ are both $2q$. As for $\C(n,q)$, we determine the minimum distance and the stopping distance for some special cases and obtain the lower bounds for other cases.

\end{abstract}

\begin{IEEEkeywords}
  LDPC code, minimum distance, stopping distance, symmetric matrix.
\end{IEEEkeywords}

\section{Introduction}\label{sec-intro}

A low-density parity-check (LDPC) code is a linear code, whose parity-check matrix $H$ satisfies the following structural properties: (1) each row consists of $\rho$ ``ones''; (2) each column consists of $\gamma$ ``ones''; (3) the number of ``ones'' in common between any two columns, denoted by $\lambda$, is no greater than 1; (4) both $\rho$ and $\gamma$ are small compared to the length of the code and the number of rows in $H$, respectively. The LDPC codes defined above are called \emph{regular}. Throughout this paper, we restrict attention to regular LDPC codes.

The concept of LDPC codes was first introduced by Gallager in 1960's~\cite{Gal62}. However, this remarkable invention was almost disregarded for the deficiency of capacity computation by researchers in the following thirty years. It was not until late 1990's and early 2000's that MacKay and Neal rediscovered LDPC codes~\cite{Mac96}. After that, a lot of efforts have been devoted to designing and constructing efficient LDPC codes. According to different constructing methods, LDPC codes can be classified into two main categories: computer-generated pseudo-random  LDPC codes and well-structured LDPC codes. In general, the performance of well-structured LDPC codes is not as good as computer-generated pseudo-random ones. But such codes are needed for implementation in the sense that their properties can be explicitly determined in many cases. The well-structured LDPC codes are constructed using various methods, e.g., finite geometry~\cite{Kou01,Tang05,FDWM15}, graphs~\cite{Kim04}, partial geometry~\cite{Van11}, combinatorial designs~\cite{Joh01,Laen07,Fal12}, generalized polygons~\cite{Liu05} and so on.

Let $\C$ be a binary $[n, k, d]$ linear code and $H$ be a parity-check matrix of the code $\C$. A \emph{stopping set} in $\C$ is the support of a binary vector having the length $n$ that does not have exactly one $1$ in common with any row of $H$. The minimum size of non-empty stopping sets of $\C$ is called the \emph{stopping distance} of the code $\C$, denoted by $s(H)$~\cite{Kas03}. It is well known that the minimum distance determines the code's error correcting capacity and the
stopping distance is an important measure of the performance of iterative decoding over the binary erasure channel. Thus, the study of the minimum distance and the stopping distance of a code is important in both theory and applications. However, in most cases, it is rather difficult to determine the minimum distance and the stopping distance of a specific LDPC code. Kou, Lin, and Fossorier~\cite{Kou01} constructed four classes of LDPC codes based on points and lines of finite
geometries and gave the lower bounds of the minimum distance. In several special cases, the minimum distance of the codes were determined. Later, Tang, Xu, Lin and Abdel-Ghaffar~\cite{Tang05} provided constructing methods based on hyperplanes of two consecutive dimensions of the finite geometries, which generalized the idea of~\cite{Kou01}. They also got the lower bounds of the minimum distance. Xia and Fu~\cite{Xia06} obtained the lower bounds of the stopping distance of these codes. Kim, Peled, Perepelitsa, Pless, and Friedland~\cite{Kim04} gave an explicit construction of two classes of LDPC codes based on $q$-regular bipartite graphs. They determined the minimum distance and the stopping distance for some special cases, and got the lower bounds for some other cases. Recently, Feng, Deng, Wang, and Ma~\cite{FDWM15} presented three classes of LDPC codes based on the totally isotropic subspaces of symplectic, unitary and orthogonal spaces over finite fields respectively and determined the minimum distance for some cases.


In this paper, we first construct a bipartite graph $G(n,q)$ with the adjacency matrix $H(n,q)$ whose vertex set contains points and maximal sets of rank $1$ of $\bs_{n}(\F_{q})$. A point is adjacent to a maximal set of rank $1$ if and only if the maximal set of rank 1 contains the point. It is easily verified that $H(n,q)$ and $H^T(n,q)$ both satisfy the four structural properties of parity-check matrices of LDPC codes. Then we obtain two classes of LDPC codes, $\C(n,q)$ and $\C^{T}(n,q)$, whose parity-check matrices are $H(n,q)$ and $H^T(n,q)$, respectively. Several properties of these two classes of LDPC codes are investigated, including the girth, the minimum distance, the stopping distance and the dimension. Some simulation results show that the constructed LDPC codes perform better than random codes.

The remainder of this paper is organized as follows: in Section~\ref{sec-pre}  we give a brief introduction to the space of $n\times n$ symmetric matrices over $\F_{q}$, and present some auxiliary results. In Section~\ref{sec-const}, we construct the two classes of binary girth-$8$ LDPC codes, $\C(n,q)$ and $\C^{T}(n,q)$. Both the minimum distance and the stopping distance of $\C^{T}(n,q)$ and $\C(n,q)$ are mainly discussed in Section~\ref{sec-main}. Some simulation results are provided in Section~\ref{sec-sim}. Finally, we conclude the paper by some open problems in Section~\ref{sec-con}.

\section{Preliminaries and Auxiliary Results}\label{sec-pre}

Let $\F_{q}$ be the finite field with $q$ elements, where $q$ is a prime power, and let $n$ be a positive integer. The set of $n \times n$ symmetric matrices over $\F_{q}$ is called \emph{the space of $n \times n$ symmetric matrices over $\F_{q}$}, denoted by $\bs_{n}(\F_{q})$. We call the elements of $\bs_{n}(\F_{q})$ the {\em points} of $\bs_{n}(\F_{q})$. Clearly, $\bs_{n}(\F_{q})$ has $q^{\frac{n(n+1)}{2}}$ points. Let $\gs_{n}(\F_{q})$ be the {\em motion group} of $\bs_{n}(\F_{q})$, i.e., the group consisting of all transformations of the form
 $$
 X\mapsto P^{T}XP+S, \quad\textrm{for all}\  X \in \bs_{n}(\F_{q})
 $$
 where $P\in \gl_{n}(\F_{q})$ and $S \in \bs_{n}(\F_{q})$. Obviously, $\gs_{n}(\F_{q})$ acts transitively on $\bs_{n}(\F_{q})$.

For two points $S, S'$ in $\bs_{n}(\F_{q})$, the \emph{arithmetic distance} between $S$ and $S'$ is defined to be the value of $\rank(S - S')$, denoted by $\ad(S, S')$. If $\ad(S, S')=1$, $S$ and $S'$ are said to be \emph{adjacent}.  When $S \neq S'$, the \emph{distance} $d(S, S')$ between $S$ and $S'$ is defined to be the least positive integer $r$ for which there is a sequence of $r+1$ points $S_{0}, S_{1}, \ldots, S_r$ with $S_0 = S$ and $S_r = S'$ such that $S_{i}$ and $S_{i+1}$ are adjacent for $0\leq i\leq r-1$. When $S=S'$, we define $d(S, S')=0$. The arithmetic distance between each pair of points of $\bs_n(\F_q)$ is preserved under the elements of the group $\gs_n(\F_q)$, and so is the distance between any two points of $\bs_n(\F_q)$~\cite{Wan96}. The following lemma establishes a link between the arithmetic distance and the distance of two points $S, S' \in \bs_{n}(\F_q)$.

\begin{lemma}\cite[Proposition 5.5]{Wan96}\label{lem-add}
For every two points $S, S' \in \bs_n(\F_q)$, $\ad(S, S') \leq d(S, S')$. More precisely, when the characteristic of $\F_q$ is not $2$, the equality sign holds, i.e., $\ad(S, S') = d(S, S')$. When the characteristic of $\F_q$ is $2$,
   \begin{equation*}
     \ad(S, S') = \left\{\begin{array}{ll}
       d(S, S') & \textrm{ if $S - S'$ is non-alternate,} \\
       d(S, S') - 1 & \textrm{ if $S - S'$ is alternate.}
     \end{array}\right.
   \end{equation*}
 \end{lemma}

We will also use the following definitions and notations. Let $S$ be a point of $\bs_{n}(\F_q)$ and $\delta$ be a positive integer. The set $$ \{S'\in \bs_{n}(\F_q): 0<d(S,  S')\leq\delta\}$$ is called \emph{the deleted $\delta$ neighbourhood of the point $S$} in $\bs_{n}(\F_{q})$, and is denoted by $U(S, \delta)$. Let $\mathbb{A}$ be a subset of $\bs_{n}(\F_q)$. Define $U(\mathbb{A},\delta)$ as the intersection of all $U(S,\delta)$ with $S\in \mathbb{A}$, i.e.
$$
U(\mathbb{A}, \delta)= \bigcap_{S \in \mathbb{A}}U(S, \delta),
$$
and $U(\mathbb{A}, \delta)$ is called the \emph{common deleted $\delta$ neighbourhood} of $\mathbb{A}$. For simplicity, let $U(S) = U(S, 1)$, $U(\mathbb{A}) = U(\mathbb{A}, 1)$,
$$\left(\begin{array}{cccc} A &0\\0&B\end{array}\right) \textrm{\ is\ denoted\ by\ }\mathrm{diag}(A,B),
$$
$$
\left(\begin{array}{cccc} a&b\\b&c\end{array}\right)_n \quad \textrm{denotes the matrix} \quad
\left(\begin{array}{cccccc}
a & b & 0 &  \cdots & 0\\
b & c & 0 & \cdots & 0\\
0 & 0 & 0 & \cdots & 0\\
\vdots & \vdots & \vdots &  \vdots & \vdots \\
0 & 0 & 0 &  \cdots & 0 \end{array}\right)_{n\times n} ,
$$
$I_{ij}$ denotes the $n\times n$ matrix whose $(i,j)$-entry is 1 and all the other entries are $0$, $I_n$ is the identity matrix of order $n$, and ${\bf 0}_{n}$ is the zero matrix of order $n$.

In what follows, we present some basic results on the properties of points and other objects in $\bs_n(\F_q)$.

\begin{lemma}\label{lem-us}
  Let $S$ be a point of $\bs_n(\F_q)$, then $|U(S)|=q^n-1$.
\end{lemma}

\begin{proof}
  Since $\gs_n(\F_q)$ acts transitively on $\bs_n(\F_q)$ and keeps invariant the arithmetic distance between any two points of $\bs_n(\F_q)$, there exists an element $\sigma$ of $\gs_n(\F_q)$ such that $\sigma(S) = {\bf 0}_n$ and $|U(S)|=|U({\bf 0}_n)|$. By definition, we have
  $$
  U({\bf 0}_n)=\{S' \in \bs_n(\F_q): 0 < d({\bf 0}_n,  S') \leq 1 \} = \{ S'\in \bs_n (\F_q): \rank (S') = 1 \}.
  $$
  Hence it suffices to show that the number of matrices of rank $1$ in $\bs_n(\F_q)$ is $q^n-1$.

  We now prove this by induction as follows. When $n = 1$, $U({\bf 0}_1)=\{(s) : s \in \F_q^{*}\}$ and the lemma holds trivially. When $n = 2$, let $S' \in \bs_2(\F_q)$ be a matrix of rank 1. Write $S'$ in the form of
  $$
  \left(\begin{array}{cc} s_{11} & s_{12}\\
    s_{12} & s_{22}
  \end{array}\right),
  $$
  where $s_{11}, s_{12}, s_{22} \in \F_q$. If $s_{11}=0$, then $s_{12}=0$ as $\rank(S')=1$. In this case, $S'$ can be regarded as a rank-$1$ matrix of $\bs_1(\F_q)$. Since $|U({\bf 0}_1)|=q-1$, it follows that there are $q-1$ different choices for $S'$. If $s_{11} \neq 0$, then the second row is an $s_{11}^{-1} s_{12}$ multiple of the first row and $s_{22} = s_{11}^{-1}s_{12}s_{12}$. Therefore, $S'$ depends only on the choices of $s_{11}$ and $s_{12}$. It follows that
  there are $q(q-1)$ different choices for $S'$ in this case. Thus, we conclude that $|U({\bf 0}_2)| = (q-1)+q(q-1) = q^{2}-1$.

  Assume now that the lemma holds for $n-1$ with $n \geq 3$, i.e., $|U({\bf 0}_{n-1})|=q^{n-1}-1$. Let $S' = (s_{ij})_{n\times n}\in \bs_n(\F_q)$ be a matrix of rank 1. If $s_{11}=0$, then $s_{1i}=0$ for all $i=2, \ldots, n$ and $S^{'}$ can then be regarded as a rank-$1$ matrix of $\bs_{n-1}(\F_q)$. Hence by induction, there are $q^{n-1}-1$ different choices for $S'$ in this case. If $s_{11}\neq 0$, then $s_{ij} = s_{11}^{-1}s_{1i}s_{1j}$ for all $1\leq i,j \leq n$,
  which implies that $S'$ depends only on the choices of $s_{11}, \ldots, s_{1n}$. Then one can see that $S'$ has $(q-1)q^{n-1}$ different choices in this case. Therefore $|U({\bf 0}_n)|=(q^{n-1}-1)+(q-1)q^{n-1}=q^n-1$. This completes the proof.

\end{proof}

We now define maximal sets of rank $1$, which also correspond to vertices (the check nodes) in the bipartite graph that we will construct later, as the same as the points of $\bs_n(\F_q)$ (the message nodes).

\begin{definition}\cite[Definition 5.3]{Wan96}\label{def-ms}
  A subset $\mathcal{M}$ of $\bs_n(\F_q)$ is called a {\em maximal set of rank $1$}, if any two points of $\mathcal{M}$ are adjacent and no point in $\bs_n(\F_q)\setminus \mathcal{M}$ is adjacent to each point of $\mathcal{M}$.
\end{definition}

Clearly, a maximal set of rank $1$ is again a maximal set of rank $1$ under actions of $\gs_n(\F_q)$. The maximal set of rank $1$ can also be interpreted in the language of graphs. Let $\Gamma(n,q)$ be an undirected graph associated with $\bs_n(\F_q)$ as its vertex set. Two vertices $S$ and $S'$ of $\Gamma(n,q)$ are adjacent if and only if $\rank (S-S')=1$. Then the vertex set of a maximal clique in the graph $\Gamma(n,q)$ corresponds to a maximal set of rank $1$ of $\bs_n(\F_q)$.

The following basic results are needed in the sequel.
\begin{proposition}\label{prop-m1}~\cite[Proposition 5.8]{Wan96}
  $\mathcal{M}_1=\{x I_{11}: x \in \F_q$\} is a maximal set of rank $1$. Moreover, any maximal set of rank 1 can be transformed to $\mathcal{M}_1$ under $\gs_n(\F_q)$.
\end{proposition}

\begin{lemma}\label{lem-s1s2}~\cite[Lemma 5.9]{Wan96}
Let $S_1, S_2$ be two points in $\bs_n(\F_q)$ with $\ad(S_1,S_2)=1$. Then there exists a unique maximal set of rank $1$ containing both $S_1$ and $S_2$, which is the set of matrices
  $$
  \{ xS_1+(1-x)S_2 : x\in \F_q \}  = \{ S_2 + x(S_1 - S_2): x \in \F_q \}.
  $$
\end{lemma}

\begin{corollary}\label{coro-inter}~\cite[Corollary 5.10]{Wan96}
  If there are two distinct maximal sets of rank $1$ whose intersection is nonempty, then their intersection contains a single point in $\bs_n(\F_q)$.
\end{corollary}

\begin{lemma}\label{lem-maxset}
  For each $S \in \bs_n(\F_q)$, there are $\frac{q^{n}-1}{q-1}$ maximal sets of rank $1$ containing $S$.
\end{lemma}

\begin{proof}
  By Lemma~\ref{lem-s1s2}, to uniquely determine a maximal set of rank $1$ containing $S$, we need to choose another point from $U(S)$, and there are $q^{n}-1$ possible choices from Lemma~\ref{lem-us}. Meanwhile, two different points, $S_{1}$ and $S_{2}$, give the same maximal set of rank 1 if and only if $S_{2}=xS_{1}+(1-x)S$ for some $x\in \F_{q}^{*}$. Thus, there are $\frac{q^{n}-1}{q-1}$ mutually distinct maximal sets of rank $1$ containing $S$.
\end{proof}

\begin{lemma}\label{lem-maxsetall}
There are $\frac{q^n-1}{q-1}q^{\frac{n^2+n-2}{2}}$ maximal sets of rank $1$ of $\bs_n(\F_q)$.
\end{lemma}

\begin{proof}
  By Lemma~\ref{lem-s1s2}, we need to choose two adjacent points, $S_1$ and $S_2$, of $\bs_n(\F_q)$. For $S_1$, there are $q^{\frac{n(n+1)}{2}}$ possible choices, then $S_2$ has $q^n-1$ possible choices from $U(S_1)$. Two choices, $(S_1, S_2)$ and $(S_1', S_2')$ give the same maximal set of rank $1$ if and only if $S_1'$ and $S_2'$ are both contained in the maximal set of rank $1$, determined by $S_1$ and $S_2$. This means that $S_1' = x_1 S_1 + (1-x_1) S_2$ and $S_2' = x_2 S_1 + (1 - x_2) S_2$ with $x_1, x_2 \in \F_q$ and $x_1 \ne x_2$. There are totally $q(q-1)$ possibilities for the choices of $x_1, x_2$. Therefore, there are $\frac{q^n-1}{q-1}q^{\frac{n^2 + n - 2}{2}}$ maximal sets of rank $1$ of $\bs_n(\F_q)$.
\end{proof}

\section{The construction of $\C(n,q)$ and $\C^T(n,q)$}\label{sec-const}

We now construct a new bipartite graph $G(n, q)$,  whose vertex set contains two parts: a point set $V(n,q) = \bs_n (\F_q)$ and a line set $L(n, q)$, i.e., the set of all the maximal sets of rank $1$ of $\bs_n(\F_q)$. A point $v \in V(n,q)$ is adjacent to a line $\ell \in L(n,q)$ if and only if the point $v$ is contained in the line $\ell$, which we denote by $v \sim \ell$. Let $H(n, q)$ be the adjacent matrix of the bipartite graph $G(n, q)$,  whose rows and
columns correspond to the lines and points of $G(n, q)$, respectively. Hence, $H(n, q)$ has $r = \frac{q^n-1}{q-1}q^{\frac{n^{2}+n-2}{2}}$ rows (by Lemma~\ref{lem-maxsetall}) and $c = q^{\frac{n(n+1)}{2}}$ columns, and it is not difficult to prove the following four structural properties:
 \begin{enumerate}
   \item[(i)] each row consists of $\rho=q$ ``ones'' (by Proposition~\ref{prop-m1});
   \item[(ii)] each column consists of $\gamma=\frac{q^{n}-1}{q-1}$ ``ones'' (by Lemma~\ref{lem-maxset});
   \item[(iii)] the number of ``ones'' in common between any two columns (rows), denoted by $\lambda$, is not greater than $1$ (by Lemma~\ref{lem-s1s2} and Corollary~\ref{coro-inter});
   \item[(iv)] both $\rho$ and $\gamma$ are small compared to the number of columns $c$ and the number of rows $r$ in $H(n,q)$, respectively.
 \end{enumerate}
 Hence the null spaces of $H(n,q)$ and its transpose $H^T(n,q)$ define two regular binary LDPC codes, denoted by $\C(n,q)$ and $\C^T(n,q)$, respectively. Based on the basic properties of points in $\bs_n(\F_q)$ we derived in Section~\ref{sec-pre}, we now investigate some properties of the bipartite graph $G(n,q)$.

 \begin{lemma}\label{lem-rank2}
   Let $S_1$, $S_2$ be two non-zero matrices of $\bs_n(\F_q)$, which are adjacent to two distinct lines passing through ${\bf 0}_{n}$, respectively. Then $\rank (S_1-S_2)=2$.
\end{lemma}

\begin{proof}
  On one hand, $S_1$ and $S_2$ are two different matrices, then $\rank (S_1 - S_2)\geq 1$; on the other hand, by rank metric we have
  $$
  \rank (S_{1}-S_{2})\leq \rank (S_{1})+\rank (S_{2})=2 .
  $$
  Thus $1 \leq \rank (S_1 - S_2) \leq 2$. We prove $\rank (S_1-S_2) = 2$ by induction on $n$.

  When $n=2$, the lines passing through $\mathbf{0}_2$ can be represented in either of the following two forms:
  \begin{eqnarray*}
  & & \left\{ x \left(\begin{array}{cc}
    1 & y \\
    y & y^2
  \end{array}\right) : x\in \F_q \right\} \quad \textrm{for $y\in \F_q$}, \\
\textrm{and}  & & \left\{ x \left(\begin{array}{cc}
    0 & 0 \\
    0 & 1
  \end{array}\right) : x\in \F_q\right\}.
\end{eqnarray*}

We discuss on the following two possible cases:

\noindent Case 1: Both $S_{1}$ and $S_{2}$ are points of the first form, then we can assume that
  $$
  S_1 = x_1 \left(\begin{array}{cc} 1 & y_{1} \\
    y_1 & y_1^2
  \end{array}\right),\quad
  S_2 = x_2  \left(\begin{array}{cc} 1 & y_2 \\
    y_2 & y_2^2
  \end{array}\right),
  $$
  where $x_1,  x_2\in \F_q^*$,  $y_1, y_2\in \F_q$ and $y_1 \neq y_2$ since $S_1, S_2$ lie on two different lines. Then we have
  $$\left|x_1  \left(\begin{array}{cc} 1 & y_1 \\
    y_1  &  y_1^2
  \end{array}\right)
  - x_{2} \left(\begin{array}{cc} 1 & y_2  \\
    y_2  &  y_2^2
  \end{array}\right)\right|
  = - x_1 x_2 (y_1 - y_2)^2 \neq 0 ,
  $$
which implies that $\rank (S_1 - S_2)=2$.

\noindent Case 2: Exactly one of $S_1$ and $S_2$ is of the first form, and the other is of the second form. Without loss of generality, we may assume that
    $$
    S_1 = x_1 \left(\begin{array}{cc}
      1 & y \\
      y & y^2
    \end{array}\right), \quad
    S_2 = x_2 \left(\begin{array}{cc}
      0 & 0 \\
      0 & 1
    \end{array}\right),
    $$
  where $x_1, x_2 \in \F_q^*$, $y \in \F_q$. In this case, we get
  $$
  \left| x_1 \left(\begin{array}{cc}
    1 & y \\
    y & y^2
  \end{array}\right) - x_2 \left(\begin{array}{cc} 0 & 0 \\
    0 & 1
  \end{array}\right)\right| = - x_1 x_2 \neq 0. $$
  It then follows that $\rank (S_1 - S_2) = 2$.

Now assume that the conclusion holds for $n-1$. Let $S_1 = (a_{ij})_{n\times n}$, and $S_2 = (b_{ij})_{n\times n}$, where $a_{ij},   b_{ij} \in \F_q$, $a_{ij} = a_{ji}$, and $b_{ij}=b_{ji}$ for $1\leq i\leq j\leq n$. Clearly, we have $\rank(S_1) = \rank(S_2) = 1$. As discussed in the proof of Lemma~\ref{lem-us}, if $a_{11}=0$, we have $a_{1i}=0$ for $1 \leq i \leq n$, then $S_1$ is regarded as a rank-$1$ matrix of $\bs_{n-1}(\F_q)$; if $a_{11}\neq 0$, then $a_{ij} = a_{11}^{-1}a_{1i}a_{1j}$ for $1 \leq i \leq j \leq n$. We discuss on the following three cases:

\noindent Case 1: If $a_{11} = b_{11} = 0$, $S_1$ and $S_2$ can be regarded as two rank-1 matrices of $\bs_{n-1} (\F_q)$ and lie on two different lines passing though $\mathbf{0}_{n-1}$. Hence $\rank (S_1 - S_2)=2$ by hypothesis.

\noindent Case 2: If one of $a_{11}$ and $b_{11}$ is 0, and the other one is nonzero, we may assume that $a_{11}=0$ and $b_{11}\neq 0$. Then
  $$
  S_2 - S_1 = \left(\begin{array}{cccc}
b_{11} & b_{12} & \cdots & b_{1n}\\
b_{12} & b_{22}-a_{22} & \cdots & b_{2n}-a_{2n}\\
\vdots & \vdots & \vdots & \vdots \\
b_{1n} & b_{2n}-a_{2n} & \cdots & b_{nn}-a_{nn}
\end{array}\right).
$$
If $\rank (S_2 - S_1) = 1$, then $b_{ij}-a_{ij} = b_{11}^{-1}b_{1i}b_{1j}$ for $2\leq i\leq j\leq n$. Since $\rank (S_{2})=1$ and $b_{11}\neq 0$, we also have $b_{ij}=b_{11}^{-1}b_{1i}b_{1j}$ for $1\leq i\leq j\leq n$. It then follows that $a_{ij}=0$ for all $1\leq i\leq j\leq n$, that is, $S_1 = \mathbf{0}_n $, contradicting to $\rank  (S_{1})=1$. Therefore, in this case $\rank ( S_2 - S_1 )=2$.

\noindent Case 3: If $a_{11} \neq 0$ and $ b_{11}\neq0$, then $a_{ij} = a_{11}^{-1}a_{1i}a_{1j}$ and $b_{ij} = b_{11}^{-1}b_{1i}b_{1j}$ for $1\leq i\leq j\leq n$, where $\rank (S_1)=\rank (S_2)=1$. Suppose that $a_{11}^{-1}a_{1i} = b_{11}^{-1}b_{1i}$ for all $1\leq i\leq n$, we get $S_2  = a_{11}^{-1}b_{11}S_1$. This means that $S_1$ and $S_2$ lie on the same line passing through $\bf0_{n}$, which is a contradiction. Hence there exists $i_0$ such that $a_{11}^{-1}a_{1i_{0}}\neq b_{11}^{-1}b_{1i_{0}}$, where $1 \leq i_0 \leq n$.
Suppose that $\rank (S_2 - S_1) = 1$. If $b_{11}-a_{11}=0$, then we obtain $b_{1i}=a_{1i}$, for all $1\leq i\leq n$. It then follows that $S_2 = S_1$, leading to a contradiction. Thus, we have $b_{11} - a_{11} \neq 0$, and then
\begin{equation*}
b_{ij}-a_{ij}=(b_{11}-a_{11})^{-1}(b_{1i}-a_{1i})(b_{1j}-a_{1j}).
\end{equation*}
On the other hand, recall that
\begin{equation*}
a_{ij}=a_{11}^{-1}a_{1i}a_{1j}\ \text{and}\ b_{ij}=b_{11}^{-1}b_{1i}b_{1j} .
\end{equation*}
With the three equations above, we obtain
$$
(a_{11}^{-1}a_{1i} - b_{11}^{-1}b_{1i})(a_{11}^{-1}a_{1j} - b_{11}^{-1} b_{1j})=0,
$$
for $1 \leq i \leq j \leq n$. When $i=j=i_{0}$, we get $(a_{11}^{-1}a_{1i_0} - b_{11}^{-1} b_{1i_0})^2=0$, which contradicts to $a_{11}^{-1}a_{1i_{0}}\neq b_{11}^{-1}b_{1i_{0}}$. Therefore, the desired conclusion is proved.

\end{proof}

Using Lemma~\ref{lem-rank2} we just proved, we are now ready to determine the girth of the bipartite graph $G(n,q)$.

\begin{theorem}\label{thm-girth}
  The girth of the bipartite graph $G(n,q)$ is 8.
\end{theorem}

\begin{proof}
  By the construction of LDPC codes (see Lemma~\ref{lem-s1s2} and Corollary~\ref{coro-inter}), we know that the girth of $G(n,q)$ is at least 6. Now we assume that there is a cycle of length $6$ in $G(n,q)$, i.e.,
  $$
  S_1 \sim \ell_1 \sim S_2 \sim \ell_2 \sim S_3 \sim \ell_3 \sim S_1.
  $$
Since the action of $\gs_n(\F_q)$ on $\bs_n(\F_q)$ is transitive, without loss of generality, we may assume that $S_1 = \mathbf{0}_n$. On one hand, $S_2$ and $S_3$ both lie on $\ell_2$, then $\rank (S_2 - S_3)=1$; On the other hand, $S_2$ and $S_3$ lie on two different lines both passing through $\mathbf{0}_n$ respectively, then $\rank (S_2-S_3) = 2$ by Lemma~\ref{lem-rank2}. This leads to a contradiction. Hence cycles of length $6$ do not exist in $G(n,q)$, which means that the girth of $G(n,q)$ is at least 8. We now give a cycle of length $8$ explicitly. It is not hard to check
 \begin{eqnarray*}
& &  \left(\begin{array}{cc}
  0 & 0 \\
  0 & 0
\end{array}\right)_n
\sim
\left\{ x \left(\begin{array}{cc}
  1 & 0 \\
  0 & 0
\end{array}\right)_n : x\in \F_q \right\}
\sim
\left(\begin{array}{cc}
  1 & 0 \\
  0 & 0
\end{array}\right)_n\\
& \sim &
\left\{\left(\begin{array}{cc}
  1 & 0 \\
  0 & 0
\end{array}\right)_n + x \left(\begin{array}{cc}
  0 & 0 \\
  0 & 1
\end{array}\right)_n : x \in \F_q\right\}
\sim
\left(\begin{array}{cc}
  1 & 0 \\
  0 & 1
\end{array}\right)_n\\
& \sim &
\left\{\left(\begin{array}{cc}
  1 & 0 \\
  0 & 1
\end{array}\right)_n + x \left(\begin{array}{cc}
  1 & 0 \\
  0 & 0
\end{array}\right)_n: x\in \F_q \right\}
\sim
\left(\begin{array}{cc}
  0 & 0 \\
  0 & 1
\end{array}\right)_n \\
& \sim &
\left\{x\left(\begin{array}{cc}
  0 & 0 \\
  0 & 1
\end{array}\right)_n: x\in \F_q\right\}
\sim
\left(\begin{array}{cc}
  0 & 0 \\
  0 & 0
\end{array}\right)_n
\end{eqnarray*}
is an 8-cycle in $G(n,q)$. This completes the proof.

\end{proof}

The LDPC codes $\C(n,q)$ and $\C^T(n,q)$ both have good girth, since belief propagation algorithms work well in general if the girth is greater than $4$. In addition, when $n = 2$, the diameter of the bipartite graph $G(2,q)$ is proved to be $6$. A bipartite graph with small diameter also allows belief propagation to work better.
\newpage
\begin{theorem}\label{thm-dia}
The diameter of the bipartite graph $G(2,q)$ is $6$.
\end{theorem}

\begin{proof}
By the definition of the diameter, we need only calculate the length $\ell(u,v)$ of the shortest path between any two distinct vertices $u$ and $v$ of $G(2,q)$. If $u$ is adjacent to $v$, then $\ell(u,v)=1$. If $u$ is not adjacent to $v$, there are three cases:

  \noindent Case 1: Both $u$ and $v$ are points. Let $r$ be the distance between $u$ and $v$ in $\mathbb{S}_{2}(\mathbb{F}_{q})$. This means that $r$ is the least positive integer for which there is a sequence of $r+1$ points $S_{0}(=u),S_{1}, S_{2},\ldots, S_{r}(=v)$ such that $S_{i}$ and $S_{i+1}$ are adjacent for $i=0,1,2,\ldots, r-1$. Then
\begin{align*}
  u = &S_{0}\sim \{S_{0}+x(S_{1}-S_{0}):x \in \mathbb{F}_{q}\}\sim \\
  &S_{1}\sim \{S_{1}+x(S_{2}-S_{1}):x \in \mathbb{F}_{q}\}\sim \\
  &\ldots \ldots \\
  &S_{r-1}\sim \{S_{r-1}+x(S_{r}-S_{r-1}):x \in \mathbb{F}_{q}\}\sim S_{r}=v
\end{align*}
is a shortest path from $u$ to $v$ in $G(2,q)$, which implies that $\ell(u,v) = 2r$. The maximal arithmetic distance between $u$ and $v$ is $\max ({\rm Rank}(u - v)) = 2$. From Lemma~\ref{lem-add} we know that when the characteristic of $\F_q$ is 2, the maximal distance between $u$ and $v$ in $\bs_2(\F_q)$ is 3, which means that the maximal $\ell(u,v)$ is 6; when the characteristic of $\F_q$ is an odd prime, the maximal distance between $u$ and $v$ in $\bs_2(\F_q)$ is 2, which means that the maximal $\ell(u,v)$ is 4.

\noindent Case 2: One of $u$ and $v$ is a point and the other is a line.
Without loss of generality, assume that $u$ is a point and $v$ is a line. Since $u$ is not adjacent to $v$, we have $\ell(u,v)\geq 3$. As $\gs_2(\F_q)$ acts transitively on $\bs_2(\F_q)$, we can assume that $u={\bf 0}_{2}$.
\begin{itemize}
  \item[(a)] When the characteristic of $\mathbb{F}_{q}$ is odd, we have $d({\bf 0}_2, S)\leq 2$ for any point $S\in \mathbb{S}_{2}(\mathbb{F}_{q})$. Let $S$ be any point on the line $v$. If $d({\bf 0}_2,S)=1$, then
$$u\sim \{u+x(S-u): x\in \mathbb{F}_{q}\}\sim S \sim v$$
is a shortest path from $u$ to $v$ in $G(2,q)$, which implies that  $\ell(u,v)=3$. Otherwise, $d({\bf0}_{2},S)=2$, which means that there exists $S'\in \mathbb{S}_{2}(\mathbb{F}_{q})$ such that $u$ is adjacent to $S'$ and $S'$ is adjacent to $S$. One can see that
$$u \sim \{u+x(S'-u): x \in \mathbb{F}_{q}\} \sim S' \sim \{S' + x(S-S'): x \in \mathbb{F}_{q}\} \sim S \sim v$$
is a path from $u$ to $v$ in $G(2,q)$, which implies that $\ell(u,v) \leq 5$.

\item[(b)] When the characteristic of $\F_q$ is $2$, let $S$ be any point of $\mathbb{S}_{2}(\mathbb{F}_{q})$ such that $d(S,{\bf 0}_{2})=3$, then $S$ is alternate, i.e.,
$$
S=\left(\begin{array}{cccc} 0&x\\x&0\end{array}\right), \textrm{ where}~ x \in \F_q^*.
$$
It is easy to see that any maximal set of rank 1 contains at most one point of the form
$$
\left(\begin{array}{cccc} 0&x\\x&0\end{array}\right), \textrm{ where}~ x \in \F_q^*.
$$
It then follows that there exists at least one point $S' \in v$ such that $d(S',{\bf0}_{2})\leq2$. Then with similar argument, we conclude that $\ell (u,v) \leq 5$.
\end{itemize}

\noindent Case 3: Both $u$ and $v$ are lines.
If $u\cap v\neq \emptyset$, then $\ell(u,v)=2$.
If $u\cap v=\emptyset$, denote by $\mathbb{A}$ the set of points on the lines that intersect with $u$ and by $\mathbb{B}$ the set of points on the lines that intersect with $v$. Note that a line of $\bs_2(\F_q)$ contains $q$ points and there are $q+1$ lines passing through a given point. Then it follows from Theorem~\ref{thm-girth} that $|\mathbb{A}|=|\mathbb{B}|=q(q(q-1))+q=q^3-q^2+q$. Thus, $\mathbb{A}\cap \mathbb{B}\neq \emptyset$ for
$|\mathbb{S}_2(\mathbb{F}_{q})|=q^3$, which implies that $\ell(u,v)\leq 6$.
Let
$$
u = \left(\begin{array}{cc}
  0 & 0 \\
  0 & 0
\end{array}\right) + x \left(\begin{array}{cc}
    1 & 0 \\
    0 & 0 \end{array}\right) \quad \textrm{for $x \in \F_q$,}
$$

$$
v = \left(\begin{array}{cc}
  0 & 1 \\
  1 & 0
\end{array}\right) + x \left(\begin{array}{cc}
  1 & 0 \\
  0 & 0
\end{array}\right) \quad \textrm{for $x \in \F_q$.}
$$

It is easy to check that $\ell(u,v)=6$.


Hence, we conclude that the diameter of $G(2,q)$ is 6.

\end{proof}

\section{The Minimum Distance and the Stopping Distance}\label{sec-main}

In this section, we investigate both the minimum distance and the stopping distance of the two classes of LDPC codes: for $\C^T(n,q)$, these two properties could be determined explicitly; for $\C (n,q)$, they could be determined for two special cases. In addition, we give lower bounds on the dimensions of $\C(n,q)$ and $\C^{T}(n,q)$ for general $n$ and $q$.

To determine the minimum distance of LDPC codes, we try to find the smallest integer $d$, for which there are $d$ linearly dependent columns in $H$. Then, any $r$ columns with $1 \leq r \leq d-1$ of $H$ are linearly independent~\cite{Lint99}. This argument will be repeatedly used when we determine the minimum distance of $\C(n,q)$ and $\C^T(n,q)$ in the section.

\subsection{The Minimum Distance and the Stopping Distance of $\C^{T}(n,q)$}\label{sec-sub1}

Both of the minimum distance and the stopping distance of $\C^T (n,q)$ can be explicitly determined for all $n, q$, as shown in the following theorem.

\begin{theorem}\label{thm-distct}
  The minimum distance and the stopping distance of $\mathcal{C}^T(n,q)$ are both $2q$.
\end{theorem}

\begin{proof}
Since the girth of $G(n,q)$ is $8$, the minimum distance of $\mathcal{C}^T (n,q)$ is at least $2q$ by~\cite[Theorem 2]{Tan81}. We now find $2q$ linearly dependent columns in $H^T (n,q)$. Consider the columns corresponding to the following two line sets:
 \begin{eqnarray*}
    \mathcal{A}_1 & = & \left\{
    \left\{\left(\begin{array}{cc}
      0 & 0 \\
      0 & x
    \end{array}\right)_n + y \left(\begin{array}{cc}
      1 & 0 \\
      0 & 0
    \end{array}\right)_n : y \in \F_q \right\}: \   x \in \F_q \right\}, \\
    \mathcal{A}_2 & = & \left\{
    \left\{ \left(\begin{array}{cc}
      x & 0 \\
      0 & 0
    \end{array}\right)_n + y \left(\begin{array}{cc}
      0 & 0 \\
      0 & 1
    \end{array}\right)_n : y \in \F_q \right\}: \  x \in \F_q\right\} .
  \end{eqnarray*}

Clearly, the $2q$ lines in the two line sets above are distinct. Therefore, the number of the columns corresponding to the chosen lines is $2q$. We now further show that these $2q$ columns are linearly dependent over $\F_2$. It is obvious that each point of the form
$$
\left(\begin{array}{cc}
  x & 0 \\
  0 & y
\end{array}\right)_n,
$$
with $x, y\in \F_q$, of $\bs_n(\F_q)$ appears exactly twice in $\mathcal{A}_1\cup \mathcal{A}_2$. Moreover, the other points do not appear in $\mathcal{A}_1 \cup \mathcal{A}_2$. The $2q$ columns corresponding to lines in $\mathcal{A}_1\cup \mathcal{A}_2$ are linearly dependent over $\F_2$. Thus, the minimum distance of $\C^T (n,q)$ is $2q$ .

Note that Tanner's proof in~\cite{Tan81} also holds for the stopping distance. Thus, $2q$ is also a lower bound on the stopping distance of $\C^T(n,q)$. From the definition of the stopping distance, it follows that the minimum distance of $\C^T (n,q)$ is an upper bound of the stopping distance. This completes the proof.

\end{proof}

\subsection{The Minimum Distance and The Stopping Distance of $\C(n,q)$}\label{sec-sub2}
For the LDPC code $\C(n,q)$, we first present the lower bounds for both the minimum distance and the stopping distance in general, and then we determine both the two properties explicitly in some special cases.

For general $n$ and $q$, we have the following lower bounds for both the minimum distance and the stopping distance.

\begin{theorem}\label{thm-gen}
 Both the minimum distance and the stopping distance of $\C(n,q)$ are at least $\frac{2(q^n-1)}{q-1}$.
\end{theorem}

\begin{proof}
  Since the girth of $G(n,q)$ is $8$, we obtain that the minimum distance of $\C(n,q)$ is at least $\frac{2(q^n-1)}{q-1}$ by~\cite[Theorem 2]{Tan81}. The proof also holds for the stopping distance.
\end{proof}

It is clear that the lower bounds are not tight for many cases. We now consider some special cases, for which we are able to determine both the minimum distance and the stopping distance explicitly. Hereafter, assume that the characteristic of $\F_q$ is $2$. When $n = 2$, both the minimum distance and the stopping distance of $\C (2,q)$ are explicitly given in the following theorem.

\begin{theorem}\label{thm-distn2q}
  Both the minimum distance and the stopping distance of $\C(2,q)$ are $4q$.
\end{theorem}

We postpone the proof of this theorem to state the following results which will be used in the proof.

\begin{lemma}\label{lem-upper4q}
Let $d$ be the minimum distance of $\C(2,q)$. Then $d\leq 4q$.
\end{lemma}

\begin{proof}
  We prove this result by choosing $4q$ columns of $H(2,q)$, which are linearly dependent over $\F_2$. Let $\alpha$ be a primitive element of $\F_q$, and consider the columns corresponding to the following points:
\begin{align*}
  P_{1} & = \left(\begin{array}{cccc} 1 & 0 \\
    0 & 0
  \end{array}\right),
  & P_{i} & =\left(\begin{array}{cccc} 1 & \alpha^{i-2} \\
    \alpha^{i-2} & \alpha^{2i-4}
  \end{array}\right), \\
  P_{q+1} & = \left(\begin{array}{cccc} 0 & 0 \\
    0 & 1
  \end{array}\right),
  & P_{q+i} & =\left(\begin{array}{cccc} 0 & \alpha^{i-2} \\
    \alpha^{i-2} & \alpha^{2i-4}+1
  \end{array}\right), \\
  Q_{1} & = \left(\begin{array}{cccc}  0 & 0 \\
    0 & 0
  \end{array}\right),
  & Q_{i} & =\left(\begin{array}{cccc}  0 & \alpha^{i-2} \\
    \alpha^{i-2} & \alpha^{2i-4}
  \end{array}\right),  \\
  Q_{q+1} & =\left(\begin{array}{cccc}  1 & 0 \\
    0 & 1
  \end{array}\right),
  & Q_{q+i} & =\left(\begin{array}{cccc}  1 & \alpha^{i-2} \\
    \alpha^{i-2} & \alpha^{2i-4}+1
  \end{array}\right),
\end{align*}
where $2\leq i\leq q$. Clearly the $4q$ points above are distinct. Now we further prove that the $4q$ columns are linearly dependent. Let $\mathbb{N}_1 = \{P_1, P_2, \ldots, P_{2q} \}$ and $\mathbb{N}_2 = \{Q_1, Q_2, \ldots, Q_{2q} \}$. Note that any two distinct points of $\mathbb{N}_i$ with $i=1, 2$ are non-collinear. Hence, it follows that no three points are collinear in $\mathbb{N}_1 \cup \mathbb{N}_2$.

Define a $(0,1)$-matrix $M$, whose rows are indexed by the points of $\mathbb{N}_1$, and whose columns are indexed by the points of $\mathbb{N}_2$. The $(i,j)$-th entry in $M$ is either $1$ or $0$ depending on whether $P_i$ and $Q_j$ are collinear or not.

Then
$$
M = \bordermatrix{
& Q_1 & \cdots & Q_{2q} \cr
   P_1 & J_q & & I_q \cr
   \vdots  \cr
   P_{2q} & I_q & & J_q
}
$$
where $I_q$ is the $q \times q$ identity matrix and $J_q$ is the $q \times q$ all-one matrix. From $M$ we know that for any point $P_i$ of $\mathbb{N}_1$, there are $q+1$ points which are collinear with $P_i$, and they are $Q_{j_{1}},Q_{j_{2}}, \ldots, Q_{j_{q+1}}$. In fact, $\{P_i + x(Q_{j_{k}}-P_i): x \in \F_{q} \}$, where $1\leq k\leq q+1$, contains all lines passing through $P_i$. Similarly, the points of $\mathbb{N}_2$ have the property above. Then each line containing one point of $\mathbb{N}_1 \cup \mathbb{N}_2$ contains exactly two points of $\mathbb{N}_1 \cup \mathbb{N}_2$. Therefore, the $4q$ columns corresponding to $\mathbb{N}_1 \cup \mathbb{N}_2$ are linearly dependent over $\F_2$. It then follows that the minimum distance $d$ of $\C(2,q)$ satisfies $d \leq 4q$.
\end{proof}

The upper bound of the minimum distance of $\C(2,q)$ is derived in Lemma~\ref{lem-upper4q}. Now we are devoted to determine the lower bound. The following lemmas discuss some basic properties about the (common) deleted neighbourhood and will be used to prove that $4q$ is also a lower bound of the minimum distance of $\C(2,q)$.

\begin{lemma}\label{lem-twocommon}
  Let $S_1, S_2$ be two distinct points of $\bs_2(\F_q)$. If $S_1$ and $S_2$ are non-collinear, then $|U(S_1)\cap U(S_2)|=0$ or $q$. In particular,
  \begin{itemize}
    \item[(a)] when $d(S_1,  S_2)=2$, $|U(S_1)\cap U(S_2)|=q$. Furthermore, $\{S_1, S_2\} \cap U(\{S_1, S_2\}) = \emptyset$ and any two points of $U(\{S_1, S_2\})$ are non-collinear;
    \item[(b)] when $d(S_{1},  S_{2})=3$, $|U(S_1)\cap U(S_2)|=0$.
  \end{itemize}
\end{lemma}

\begin{proof}
  The transformation $\sigma_1$: $X \mapsto X-S_1$, is an element of $\gs_2(\F_q)$, and transforms $S_1$ and $S_2$ to $\mathbf{0}_2$ and $S_2 - S_1$, respectively. Since $\rank (S_2-S_1) = 2$, there is a
$P \in \gl_2(\F_q)$ such that
$$
P^T (S_2 - S_1) P = \left(\begin{array}{cccc}
  1 & 0 \\
  0 & 1
\end{array}\right)
\textrm{ or }
\left(\begin{array}{cccc}
  0 & 1 \\
  1 & 0
\end{array}\right).
$$

Let $\sigma_{2}$: $X \mapsto P^T X P$,  which is also an element of $\gs_2(\F_q)$. Then we have
$$
\sigma_2 (\sigma_1 (S_1)) = \mathbf{0}_2
\textrm{ and }
\sigma_2( \sigma_1 (S_2))=\left(\begin{array}{cccc}
  1 & 0 \\
  0 & 1
\end{array}\right)
\textrm{ or }
\left(\begin{array}{cccc}
  0 & 1 \\
  1 & 0
\end{array}\right).
$$

\begin{itemize}
  \item[(a)] When $d(S_{1}, S_{2})=2$, without loss of generality, we may assume that
$$
S_1 = \left(\begin{array}{cccc}
  0 & 0 \\
  0 & 0
\end{array}\right),
\quad S_2 = \left(\begin{array}{cccc}
  1 & 0 \\
  0 & 1
\end{array}\right).
$$
Let
$$
\left(\begin{array}{cccc}
  x & y \\
  y & z
\end{array}\right) \in U(S_1) \cap U(S_2)$$
where $x,  y, z \in \F_q$. Thus, by definition, we have
$$
\left|\begin{array}{cccc}
  x & y \\
  y & z
\end{array}\right| = 0  ~\text{and}~
\left|\begin{array}{cccc}
  x-1 & y \\
  y & z-1
\end{array}\right| = 0.
$$
It then follows that
\begin{equation}\label{eqn-qequ}
z = x + 1, \quad y^2 = x(x+1).
\end{equation}
Since the characteristic of $\F_{q}$ is 2, for each $x \in \F_q$, we can uniquely determine $y$ and $z$ that satisfy (\ref{eqn-qequ}). Therefore $|U(S_1) \cap U(S_2)| = q$. Since $S_1$ and $S_2$ are not collinear, it is clear that $\{S_1, S_2 \} \cap U( \{S_1, S_2 \}) = \emptyset$. Now assume by contradiction that there exist two collinear points $P_1, P_2$ of $U(\{S_1, S_2\})$. Then there is a cycle of length $6$: $S_1 \sim
\{S_1 + x (P_1 - S_1): x \in \F_q\} \sim P_1 \sim \{P_1 + x(P_2 - P_1): x \in \F_q\} \sim P_2 \sim \{P_2 + x(S_1 - P_2): x \in \F_q \} \sim S_1$. This contradicts to Theorem~\ref{thm-girth}. Thus, any two points of $U(\{S_1, S_2\})$ are non-collinear.

\item[(b)] When $d(S_1, S_2)=3$, without loss of generality, we assume that
$$
S_1 = \left(\begin{array}{cccc}
  0 & 0 \\
  0 & 0
\end{array}\right),
\quad
S_2 = \left(\begin{array}{cccc}
  0 & 1 \\
  1 & 0
\end{array}\right).
$$
 Let
 $$
 \left(\begin{array}{cccc}
   x & y \\
   y & z
 \end{array}\right) \in U(S_1) \cap U(S_2) ,
 $$
 where $x,  y, z \in \F_q$. Then we have
 $$
 \left|\begin{array}{cccc}
   x & y \\
   y & z
 \end{array}\right| = 0 ~\text{and}~
 \left|\begin{array}{cccc}
   x & y-1 \\
   y-1 & z
 \end{array}\right| = 0.
 $$
 In this case, we get that $1=0$, which is impossible. Hence, $|U(S_1) \cap U(S_2)| = 0$.
\end{itemize}
\end{proof}

After discussing the common deleted neighbourhood of two non-collinear points, now we consider the property of three non-collinear points.

\begin{lemma}\label{lem-threecommon}
  Let $S_1, S_2, S_3$ be three distinct points of $\bs_2(\F_q)$. Suppose that they are mutually non-collinear. Then $|U(\{S_1, S_2, S_3\})|=0, 1$ or $q$. In particular, if the distance between any two of them is $2$, then $\{S_1, S_2, S_3\} \cap U(\{S_1, S_2, S_3\}) = \emptyset$ and any two points of $U(\{S_1, S_2, S_3\})$ are non-collinear.
\end{lemma}

\begin{proof}
  We consider the following two cases:

  \noindent Case 1:  There exist two points $S_i, S_j$ whose distance is $d(S_{i}, S_{j})=3$, where $1\leq i < j \leq 3$.

  By Lemma~\ref{lem-twocommon}, we get $U(\{S_i, S_j\}) = \emptyset$. Clearly, $U( \{ S_1, S_2, S_3 \} )\subseteq U(\{S_i, S_j\})$. Thus,  $|U(\{S_1, S_2, S_3\})|$ $ = 0$.

\noindent Case 2: The distance between any two points of $\{ S_1, S_2, S_3 \}$ is $2$.

According to Lemma~\ref{lem-twocommon}, without loss of generality, we may assume
$$
S_1 = \left(\begin{array}{cccc}
  0 & 0 \\
  0 & 0
\end{array}\right),
S_2 = \left(\begin{array}{cccc}
  1 & 0 \\
  0 & 1
\end{array}\right),
S_3 = \left(\begin{array}{cccc}
  a & b \\
  b & c
\end{array}\right),
$$
where $a, b, c \in \F_q, (a,c)\neq (0,0)$ and $(a,c)\neq (1,1).$ Notice that
$$
U(\{S_1, S_2, S_3\}) = \left\{\left(\begin{array}{cccc}
  x & y \\
  y & x+1
\end{array}\right) : y^2 = x(x+1),  x, y\in \F_q\right\}\cap U(S_3).
$$

By the definition of common deleted neighborhood, we have
$$
\left|\left(\begin{array}{cccc}
  x & y \\
  y & x+1
\end{array}\right)-\left(\begin{array}{cccc}
  a & b \\
  b & c
\end{array}\right)\right| = (a+c)x+a+ac+b^{2}=0.$$
Hence we need only to determine the number of solutions of this equation.
\begin{itemize}
  \item If $a=c$ and $a+ac+b^2 = 0$, then $x$ runs through $\F_q$. In this case, $|U(\{S_1, S_2, S_3\})|=q$;
  \item If $a=c$ and $a+ac+b^{2}\neq0$, then the equation cannot be solved for $x$ in $\F_{q}$. It follows that $|U(\{S_1, S_2, S_3\})|=0$;
  \item If $a\neq c$, then $x$ is uniquely determined. Hence in this case $|U(\{S_1, S_2, S_3\})|=1$.
\end{itemize}

The rest part of the proof is similar to that of Lemma~\ref{lem-twocommon} and is thus omitted here. This completes the proof of the lemma.
\end{proof}

\begin{remark}\label{rmk-common}
  If $|U(\{S_1, S_2, S_3\})| = q$, where $q > 2$ is a power of $2$, let $U(\{S_1, S_2, S_3\})=\{P_1, P_2, \ldots, P_q \}$ and choose arbitrarily three distinct points of this set, say $P_i, P_j$ and $P_k$, where $1 \leq i, j, k\leq q$. By the definition of the common deleted neighbourhood, $\{S_1, S_2, S_3\} \subseteq U(\{P_i, P_j, P_k \})$. From Lemma~\ref{lem-threecommon}, it follows that $|U(\{P_i, P_j, P_k \})| = q$. On the other hand, we have $U(\{P_i, P_j, P_k\}) \subseteq U(\{P_i, P_j\})$ and
  $|U(\{P_i, P_j\})|=q$ (by Lemma~\ref{lem-twocommon}). Hence, $U(\{P_i, P_j, P_k\}) = U(\{P_i, P_j\})$. Since $k$ can be chosen arbitrarily from $\{1,2,\ldots,q\}\setminus\{i,j\}$, we conclude that  $U(\{P_1, P_2,\ldots, P_q \}) = U(\{P_i, P_j \})$. This indicates that, if a point is collinear with any two points of $\{ P_1, P_2,\ldots, P_q \}$, then the point is collinear with every point of $\{ P_1, P_2,\ldots, P_q \}$. Let $U(\{P_i, P_j\}) = \{S_1, S_2,\ldots, S_q\}$, then it is checked that this property also holds for the set $\{S_1, S_2,\ldots, S_q\}$. It is obvious that $\{P_1, P_2, \ldots, P_q \} \cap \{S_1, S_2,\ldots, S_q\} = \emptyset$. With similar argument as in the proof of Lemma~\ref{lem-twocommon}, we can show that any two points of $\{P_1, P_2, \ldots, P_q \}$ are non-collinear, and the same is true for $\{S_1, S_2,\ldots, S_q\}$.
\end{remark}

In Lemma~\ref{lem-upper4q}, we have found an upper bound of the minimum distance $\C(2,q)$. Now we are ready to derive its lower bound.

\begin{lemma}\label{lem-lower4q}
  Let $\mathbb{A}$ be the set of points that corresponds to an arbitrarily chosen nonempty set of linearly dependent columns of $H(2,q)$. Then $|\mathbb{A}|\geq 4q$.
\end{lemma}

Since the proof of this lemma is lengthy, we put it in Appendix.

As an immediate consequence of Lemma~\ref{lem-upper4q} and Lemma~\ref{lem-lower4q}, the minimum distance of $\C(2,q)$ is $4q$. We are now ready to present the proof of Theorem~\ref{thm-distn2q}.

\begin{proof}[Proof of Theorem~\ref{thm-distn2q}]
The conclusion on the minimum distance directly follows from Lemma~\ref{lem-upper4q} and Lemma~\ref{lem-lower4q}. By the definition of the stopping distance, we obtain that the minimum distance of $\C(2,q)$ is an upper bound of the stopping distance. On the other hand, the proof in Lemma~\ref{lem-lower4q} also holds for an arbitrary stopping set of $\C(2,q)$. Consequently, we get the lower bound $4q$ of the stopping distance. This completes the proof.
\end{proof}

The discussion above determines the minimum distance and the stopping distance of $\C(2,q)$, where $q$ is a power of $2$. However, when $q$ is an odd prime power, the situation is more complicated, and we cannot explicitly determine these two properties of $\C(2,q)$. Using magma programs, it is verified that the minimum distance of $\C(2,3)$ is $12$ and the minimum distance of $\C(2,5)$ is $20$. For bigger $q$, it is still mysterious and even impossible to check by magma programs. For general $n$ and $q$, some lower bounds are obtained (see Theorem~\ref{thm-gen}).

When $q = 2$, for general $n$, we have the following theorem on both the minimum distance and the stopping distance. This indicates that $\C(n,2)$ contains only two codewords: the all-zero and the all-one codewords.

\begin{theorem}\label{thm-nq2}
  The minimum distance and the stopping distance of $\C(n,2)$ are both $2^{\frac{n(n+1)}{2}}$.
 \end{theorem}

 \begin{proof}
    Each row of $H(n,2)$ has 2 ``ones'', so all columns of $H(n,2)$ are linearly dependent over $\F_2$. Therefore, the minimum distance of $\C(n,2)$ is at most $2^{\frac{n(n+1)}{2}}$.

    We complete the proof by proving that $\Gamma(n,2)$ is connected. Recall that $\Gamma(n,2)$ is an undirected graph with $\bs_n (\F_2)$ as its vertex set. Two vertices $S$ and $S'$ of $\Gamma$ are adjacent if and only if $\rank (S-S')=1$. To prove the connectness of $\Gamma(n,2)$, it is sufficient to prove that there exist a path for any two vertices $S, S'$ of $\Gamma(n,2)$. If $S$ is adjacent to $S'$, then $S, S'$ is the path connecting $S$ and $S^{'}$. If $S$ and $S'$ are two non-adjacent vertices, suppose that $\rank(S' - S) = r > 1$. If $r$ is odd, then there exists a $P\in \gl_n (\F_2)$ such that
    $$
    P^{T}(S' - S)P = \left(\begin{array}{cccc}
      I_{r} & \\
      &  \mathbf{0}_{n-r}
    \end{array}\right).
    $$
    Let $\sigma_1: X \mapsto X-S$ and $\sigma_2: X \mapsto P^TXP$. Then $\sigma_1$ and $\sigma_2$ are both elements of $\gs_n (\F_2)$ and
    $$
    (\sigma_{2}\sigma_{1})(S)  =  \mathbf{0}_n \quad \textrm{and}\quad  (\sigma_{2}\sigma_{1})(S') = \mathrm{diag}(I_r, \mathbf{0}_{n-r}).
    $$
    $$
    S = (\sigma_{2}\sigma_{1})^{-1}(\mathbf{0}_n),\quad (\sigma_{2}\sigma_{1})^{-1}(\mathrm{diag}(I_1, \mathbf{0}_{n-1})), \ldots,(\sigma_{2}\sigma_{1})^{-1}(\mathrm{diag}(I_r, \mathbf{0}_{n-r})) = S'
    $$
    is the path from $S$ to $S'$. If $r$ is even, then there exists an element $P\in \gl_n(\F_2)$ such that
    $$
    P^{T}(S-S')P = \left(\begin{array}{cccc}
      I_r  &   \\
      &   \mathbf{0}_{n-r}
    \end{array}\right)\quad \textrm{or} \quad
    \left(\begin{array}{cccccccc}
      0 & 1 \\
      1 & 0 \\
      & & \ddots \\
      & & & 0 & 1 \\
      & & & 1 & 0 \\
      & & & & & \mathbf{0}_{n-r}
    \end{array}\right).
      $$
    Notice that
    $$
    \mathbf{0}_n,\quad
    \left(\begin{array}{cccc}
      1 & 1 \\
      1 & 1
    \end{array}\right)_{n},\quad
    \left(\begin{array}{cccc}
      0 & 1 \\
      1 & 1
    \end{array}\right)_{n}, \quad
    \left(\begin{array}{cccc}
      0 & 1 \\
      1 & 0 \end{array}\right)_{n}
      $$
      is a path containing
      $$
      \mathbf{0}_{n} \textrm{ and }
      \left(\begin{array}{cccc}
	0 & 1 \\
	1 & 0
      \end{array}\right)_{n}.
      $$
    By this fact and using the method above, we can always find a path from $S$ to $S'$ in this case. Thus, $\Gamma(n,2)$ is connected.

    Let $\mathbb{A}$ be the non-empty point set corresponding to a set of linearly dependent columns of $H(n,2)$ over $\F_2$. By the facts that $\Gamma(n,2)$ is connected and every row of $H(n,2)$ has 2 ``ones'', we obtain that $\mathbb{A}$ must contain all points of $\bs_n(\F_2)$. This means that $2^{\frac{n(n+1)}{2}}$ is a lower bound of the minimum distance of $\C(n,2)$. Hence, the minimum distance of $\C(n,2)$ is $2^{\frac{n(n+1)}{2}}$. The result on the stopping distance follows from that on the minimum distance similar to the proof of Theorem~\ref{thm-distn2q}.

  \end{proof}

\begin{example}\label{ex-22}
Let $n=2, q=2$. The vertex set of the bipartite graph $G(2,2)$ is  $V(2,2) \cup L(2,2)$.

\noindent The point set $V(2,2)$ is
\begin{align*}
  V(2, 2)=\Bigg\{
    &v_1=\left(\begin{array}{cc} 0&0\\0&0\end{array}\right),
    v_2=\left(\begin{array}{cc} 1&0\\0&0\end{array}\right),
    v_3=\left(\begin{array}{cc} 1&1\\1&1\end{array}\right),
    v_4=\left(\begin{array}{cc} 0&0\\0&1\end{array}\right), \\
    &v_5=\left(\begin{array}{cc} 0&1\\1&1\end{array}\right),
    v_6=\left(\begin{array}{cc} 1&0\\0&1\end{array}\right),
    v_7=\left(\begin{array}{cc} 1&1\\1&0\end{array}\right),
    v_8=\left(\begin{array}{cc} 0&1\\1&0\end{array}\right)
     \Bigg\},
\end{align*}
The line set $L(2,2)$ is
\begin{align*}
  L(2, 2)=\Bigg\{
    &\ell_1=\left\{\left(\begin{array}{cc} 0&0\\0&0\end{array}\right),
    \left(\begin{array}{cc} 1&0\\0&0\end{array}\right)\right\},\quad
    \ell_2=\left\{\left(\begin{array}{cc} 0&0\\0&0\end{array}\right),
    \left(\begin{array}{cc} 1&1\\1&1\end{array}\right)\right\}, \\
    & \ell_3=\left\{\left(\begin{array}{cc} 0&0\\0&0\end{array}\right),
    \left(\begin{array}{cc} 0&0\\0&1\end{array}\right)\right\},\quad
    \ell_4=\left\{\left(\begin{array}{cc} 1&0\\0&0\end{array}\right),
    \left(\begin{array}{cc} 0&1\\1&1\end{array}\right)\right\}, \\
    & \ell_5=\left\{\left(\begin{array}{cc} 1&0\\0&0\end{array}\right),
    \left(\begin{array}{cc} 1&0\\0&1\end{array}\right)\right\},\quad
    \ell_6=\left\{\left(\begin{array}{cc} 1&1\\1&1\end{array}\right),
    \left(\begin{array}{cc} 0&1\\1&1\end{array}\right)\right\}, \\
    & \ell_7=\left\{\left(\begin{array}{cc} 1&1\\1&1\end{array}\right),
    \left(\begin{array}{cc} 1&1\\1&0\end{array}\right)\right\},\quad
    \ell_8=\left\{\left(\begin{array}{cc} 0&0\\0&1\end{array}\right),
    \left(\begin{array}{cc} 1&0\\0&1\end{array}\right)\right\}, \\
    & \ell_9=\left\{\left(\begin{array}{cc} 0&0\\0&1\end{array}\right),
    \left(\begin{array}{cc} 1&1\\1&0\end{array}\right)\right\},\quad
    \ell_{10}=\left\{\left(\begin{array}{cc} 0&1\\1&1\end{array}\right),
    \left(\begin{array}{cc} 0&1\\1&0\end{array}\right)\right\}, \\
    & \ell_{11}=\left\{\left(\begin{array}{cc} 1&0\\0&1\end{array}\right),
    \left(\begin{array}{cc} 0&1\\1&0\end{array}\right)\right\},\quad
    \ell_{12}=\left\{\left(\begin{array}{cc} 1&1\\1&0\end{array}\right),
    \left(\begin{array}{cc} 0&1\\1&0\end{array}\right)\right\}
    \Bigg\}.
\end{align*}
Then the adjacent matrix $H(2,2)$ of the bipartite graph $G(2,2)$ is
$$
H(2, 2) = \bordermatrix{
 & v_1 & v_2 & v_3 & v_4 & v_5 & v_6 & v_7 & v_8 \cr
   \ell_1 & 1 & 1 & 0 & 0 & 0 & 0 & 0 & 0 \cr
   \ell_2 & 1 & 0 & 1 & 0 & 0 & 0 & 0 & 0 \cr
   \ell_3 & 1 & 0 & 0 & 1 & 0 & 0 & 0 & 0 \cr
   \ell_4 & 0 & 1 & 0 & 0 & 1 & 0 & 0 & 0 \cr
   \ell_5 & 0 & 1 & 0 & 0 & 0 & 1 & 0 & 0 \cr
   \ell_6 & 0 & 0 & 1 & 0 & 1 & 0 & 0 & 0 \cr
   \ell_7 & 0 & 0 & 1 & 0 & 0 & 0 & 1 & 0 \cr
   \ell_8 & 0 & 0 & 0 & 1 & 0 & 1 & 0 & 0 \cr
   \ell_9 & 0 & 0 & 0 & 1 & 0 & 0 & 1 & 0 \cr
   \ell_{10} & 0 & 0 & 0 & 0 & 1 & 0 & 0 & 1 \cr
   \ell_{11} & 0 & 0 & 0 & 0 & 0 & 1 & 0 & 1  \cr
   \ell_{12} & 0 & 0 & 0 & 0 & 0 & 0 & 1 & 1
}$$

Applying elementary transformation on $H(2,2)$, we obtain $\rank (H(2,2))=7$. The columns of $H^{T}(2,2)$ corresponding to $\ell_1, \ell_3, \ell_5, \ell_8$, are linearly dependent over $\F_{2}$. It is not hard to verify that any less than $4$ columns are linearly independent over $\F_{2}$. So $\C^{T}(2, 2)$ is a $[12, 5, 4]$ linear code. It is evident that $H(2,2)$ has $8$ linearly dependent columns, but any less than $8$ columns are linearly independent, so $\C(2,2)$ is an $[8, 1, 8]$ code.
\end{example}

\subsection{The Dimensions of $\C(n,q)$ and $\C^{T}(n,q)$}\label{sec-dim}

In this section, we give upper bounds on the dimensions of the LDPC codes $\C(n,q)$ and $\C^{T}(n,q)$ for general $n$ and $q$.

\begin{theorem}\label{thm-dim}
  $\C(n,q)$ is a $[q^{\frac{n(n+1)}{2}}, k_1]$ code with $k_1 \leq q^{\frac{n^2 - n}{2}}(q-1)^n$ and $\C^{T}(n,q)$ is a $[\frac{q^n-1}{q-1}q^{\frac{n^{2}+n-2}{2}}, k_2]$ code with $k_2 \leq \frac{q^n-1}{q-1} q^{\frac{n^2 + n - 2}{2}} - q^{\frac{n(n+1)}{2}} + q^{\frac{n^2 - n}{2}} (q-1)^n$.
\end{theorem}

\begin{proof}
  We prove this theorem by showing that the rank of $H(n,q)$ is at least $q^{\frac{n^2 - n}{2}} \left( q^n - (q-1)^n \right)$. Let $S=(a_{ij})_{n\times n}$ be any element of $\bs_n(\F_q)$. For each $1\leq i\leq n$, define a subset $L_i$ of the line set $L(n,q)$ as follows:
  $$
  L_i = \{\ell(S,i): S\in \bs_n(\F_q) \textrm{ and } a_{ii}=0\},
  $$
  where $\ell(S,i)=\{S+xI_{ii}: x \in \F_q\}$. Notice that each $L_i$ contains $q^{\frac{n^2+n-2}{2}}$ lines, of which any two do not intersect. This implies that the rows corresponding to any nonempty subset of $L_i$ are linearly independent over $\F_2$.

  Take $L'_2 = \{\ell(S,2): S\in \bs_n(\F_q),~a_{22}=0 \text{ and } a_{11} \neq 0\}$. We claim that the rows corresponding to any nonempty subset of $L_1\cup L'_2$ are linearly independent over $\F_2$. Suppose on the contrary that there are $k(\geq 2)$ lines of $L_1\cup L'_2$, such that the rows corresponding to these $k$ lines are linearly dependent. Denote by $L$ the set of these $k$ lines. Then there is at least one line in $L$, say $\ell_2$, contained in $L'_2$. Let $\ell_2 =
  \ell(S,2)$, where $S\in \bs_n(\F_q)$, $a_{22}=0$ and $a_{11} \neq 0$. Then by linear dependence and the property of $L_2$, there must be a line $\ell_1\in L_1\cap L$ such that $S\in \ell_1$. However, it is easy to see that the point $S-a_{11}I_{11}\in \ell_1$ lies on no lines of $L_1\cup L'_2$ other than $\ell_1$. This leads to a contradiction to the linear dependence of rows corresponding to the $k$ lines of $L$.

  Now take $L'_3=\{\ell(S,3): S \in \bs_n(\F_q),~a_{33}=0,~a_{22}\neq 0  \textrm{ and } a_{11}\neq 0\}$. With similar argument as above, we show that the rows corresponding to any nonempty subset of $L_1\cup L'_2\cup L'_3$ are linearly independent over $\F_2$. Inductively, for $3 < i \leq  n$, we can define $L'_i = \{\ell(S,i): S\in \bs_n(\F_q), ~ a_{ii}=0~\text{and}~a_{jj}\neq 0 \textrm{ for } 1\leq j < i \}$ and show that the rows corresponding to any nonempty subset of
  $\cup_{j=1}^{i}L'_j$ are linearly independent over $\F_2$, where $L'_1=L_1$.

  In fact, the rows corresponding to any nonempty subset of $\cup_{i=1}^{n} L'_i$ are linearly independent over $\F_2$, which means that the rank of $H(n,q)$ is at least $\sum_{i=1}^n |L'_i| = \sum_{i=1}^n(q-1)^{i-1}q^{\frac{n^2+n}{2} - i} = q^{\frac{n^2-n}{2}} (q^n-(q-1)^n)$. The proof is thus completed.

\end{proof}

\section{Simulation Results}\label{sec-sim}

In this section, we present some simulation results of our codes with the help of LDPC simulation facility provided by Magma~\cite{BCP97}. All the simulations were conducted over the additive white Gaussian noise (AWGN) channel, comparing the performance in terms of both the accumulated word-error rate (WER) and the accumulated bit-error rate (BER) in $50,000$ transmissions. Three LDPC codes constructed by our method, $\C^T(2,2), \C(2,4), \C^T(2,4)$, were compared with randomly
generated Gallager Codes $R[12,5], R[64,18], R[80,34]$ (also provided by Magma), respectively. Each compared random code has the same length as our code and has either the same dimension as our code or one less.

It is seen that our codes perform better than random codes with respect to both WER and BER. We remark that our codes are constructive and are possibly well studied in terms of determined properties, e.g., the minimum distance, the stopping distance, etc.

\begin{figure}[htbp]\label{fig-ct22}
\centering
  \begin{minipage}[h]{.4\textwidth}
  \centering
    \includegraphics[width=.8\textwidth]{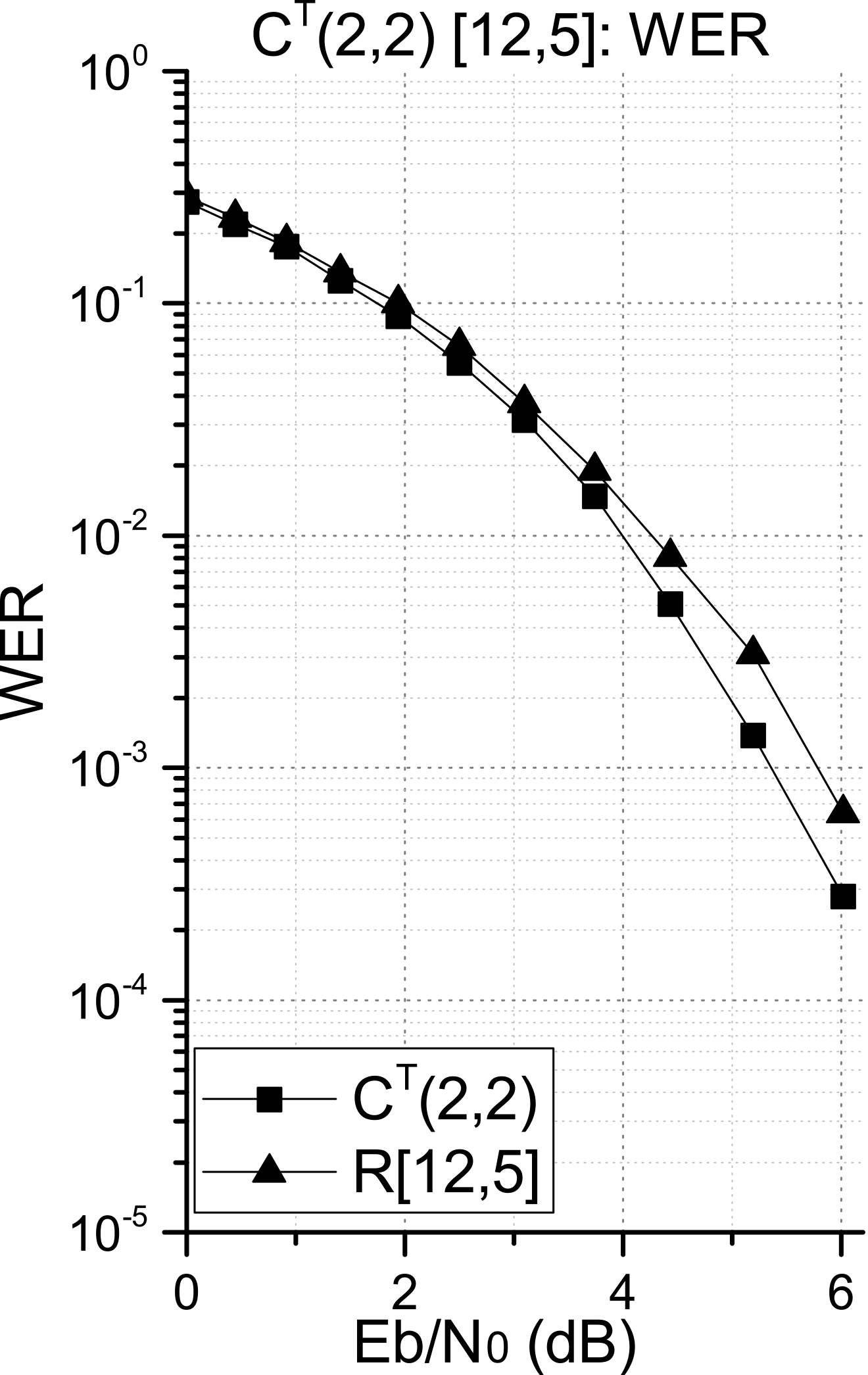}
  \end{minipage}
  \begin{minipage}[h]{.4\textwidth}
  \centering
    \includegraphics[width=.8\textwidth]{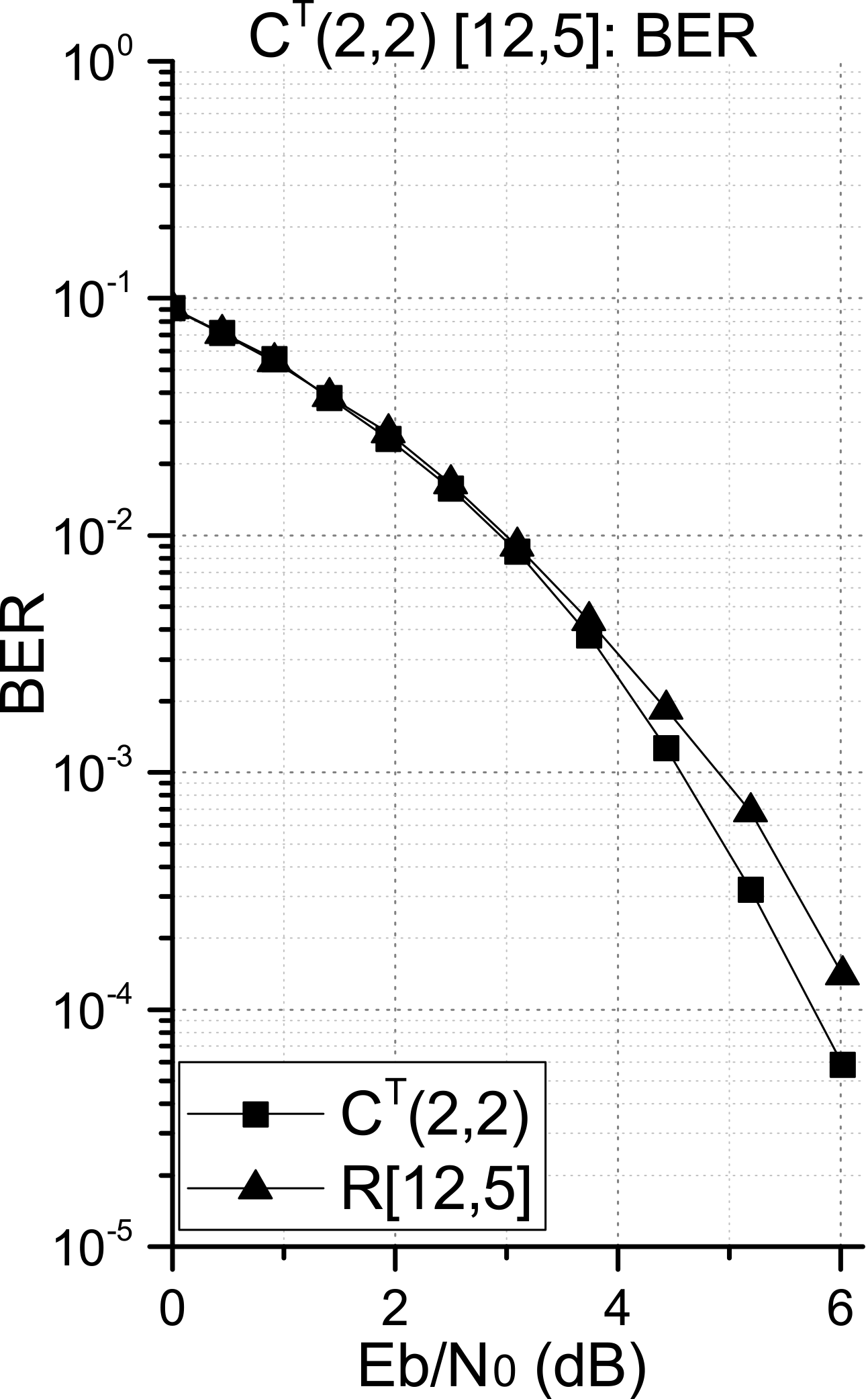}
  \end{minipage}
  \begin{minipage}[h]{.12\linewidth}
  \end{minipage}
  \caption{$\C^T(2,2)$ is a $[12,5]$ code, $R[12,5]$ is a randomly generated $(2,3)$-regular LDPC code.}
\end{figure}

\begin{figure}[htbp]\label{fig-c24}
\centering
  \begin{minipage}[h]{.4\textwidth}
  \centering
    \includegraphics[width=.8\textwidth]{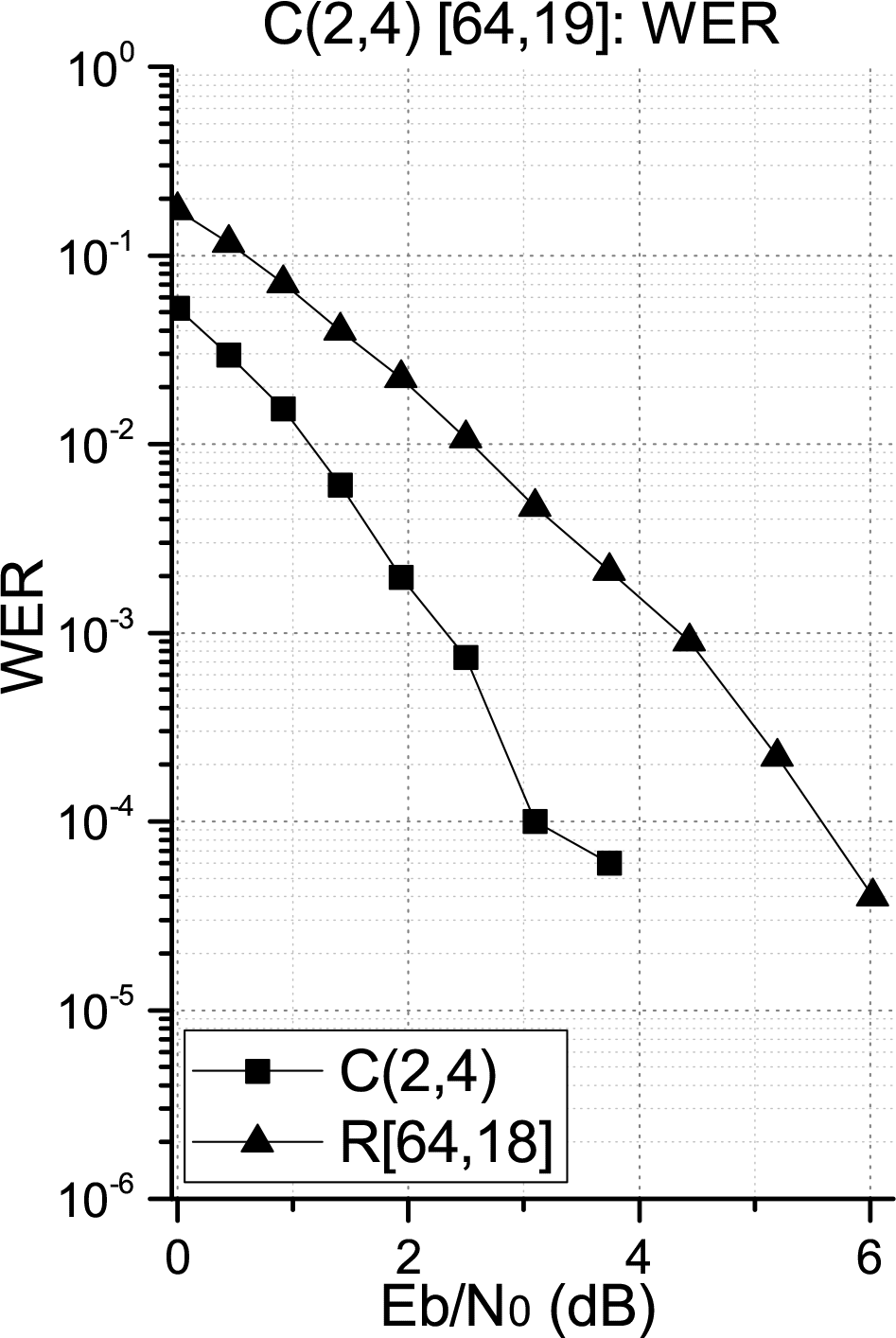}
  \end{minipage}
  \begin{minipage}[h]{.4\textwidth}
  \centering
    \includegraphics[width=.8\textwidth]{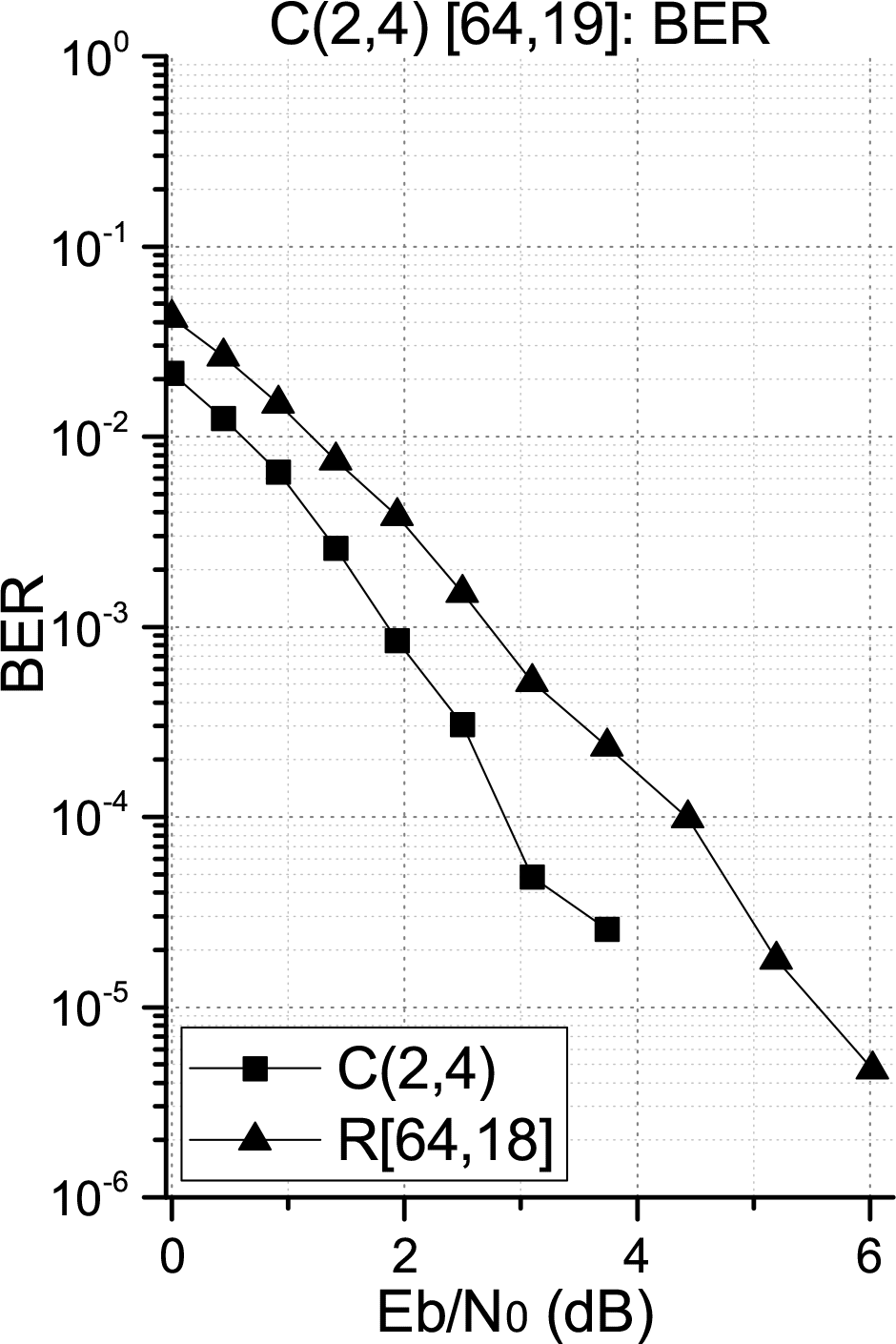}
  \end{minipage}
  \begin{minipage}[h]{.12\linewidth}
  \end{minipage}
  \caption{$\C(2,4)$ is a $[64,19]$ code, $R[64,18]$ is a randomly generated $(4,3)$-regular LDPC code.}
\end{figure}

\begin{figure}[htbp]\label{fig-ct24}
\centering
  \begin{minipage}[h]{.4\textwidth}
  \centering
    \includegraphics[width=.8\textwidth]{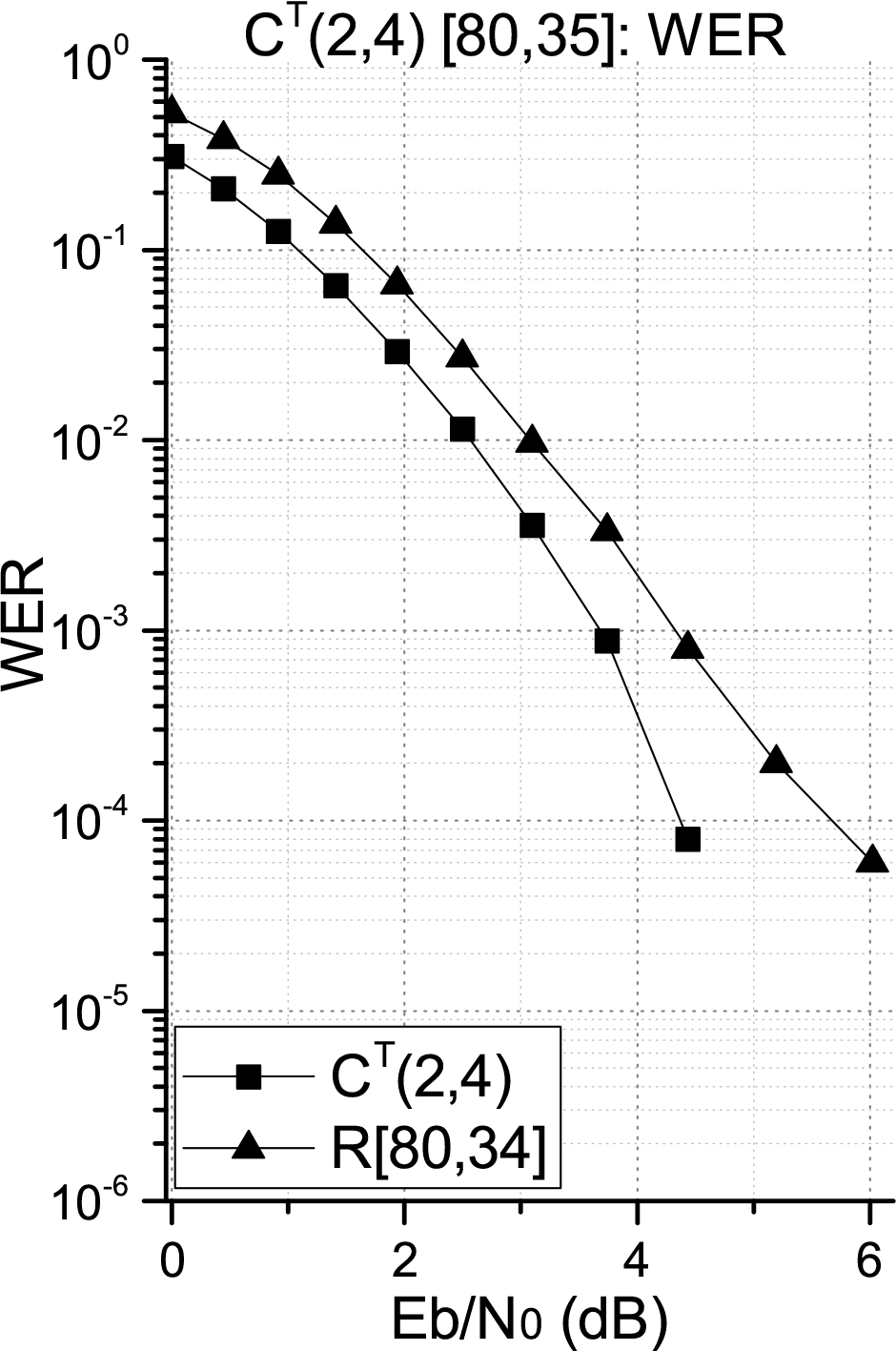}
  \end{minipage}
  \begin{minipage}[h]{.4\textwidth}
  \centering
    \includegraphics[width=.8\textwidth]{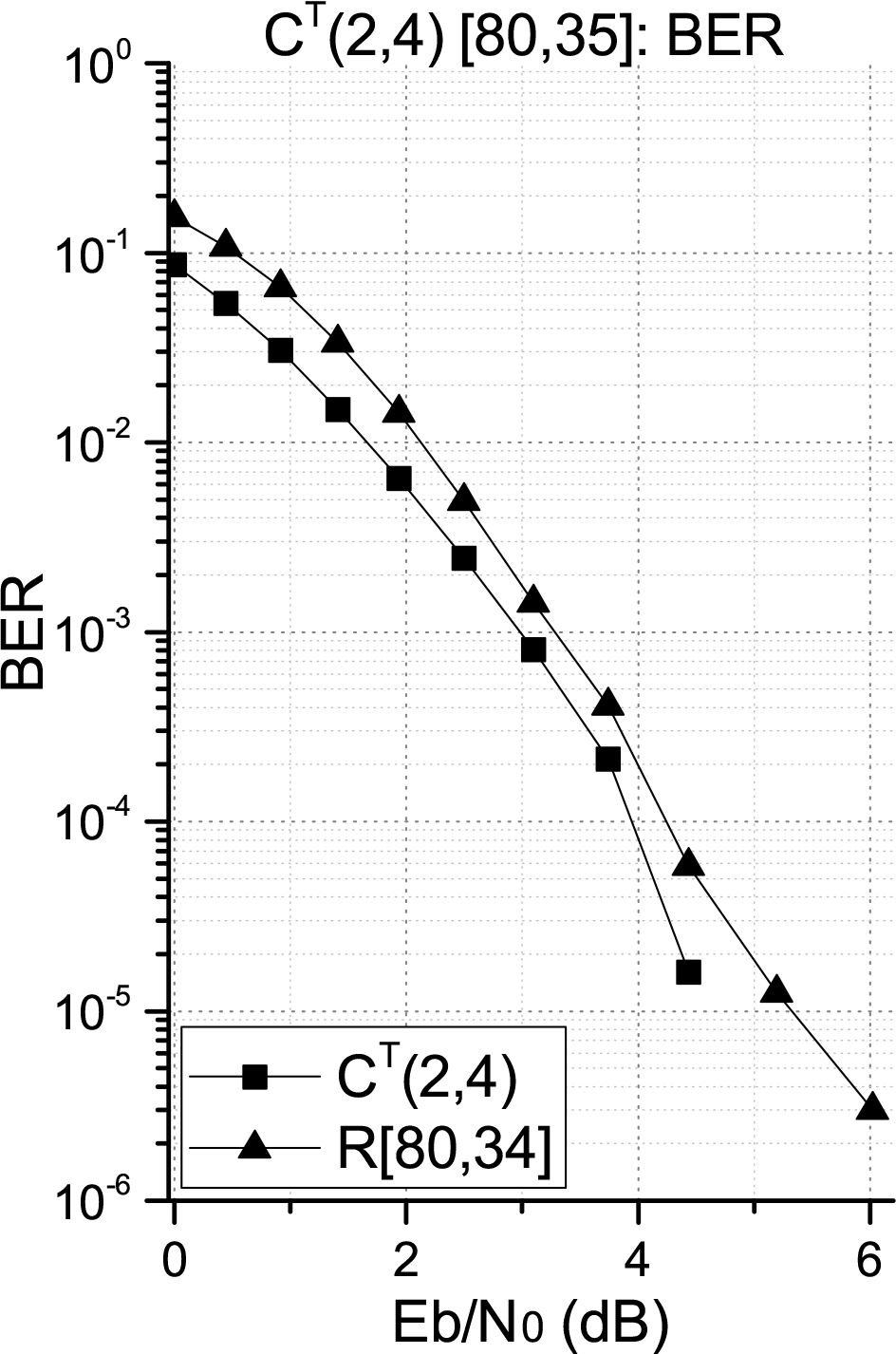}
  \end{minipage}
  \begin{minipage}[h]{.12\linewidth}
  \end{minipage}
  \caption{$\C^T(2,4)$ is a $[80,35]$ code, $R[80,34]$ is a randomly generated $(5,3)$-regular LDPC code.}
\end{figure}

\section{Concluding remarks}\label{sec-con}

We proposed a new approach to constructing LDPC codes using the space of symmetric matrices over $\F_q$. While both the minimum distance and the stopping distance of $\C^T(n,q)$ have been determined, it remains a challenge to determine these two properties of $\C(n,q)$ for general $n$ and $q$. We note that the same method can be applied to the space of $n\times n$ alternate matrices over $\F_{q}$ and the space of $n\times n$ Hermitian matrices over $\F_{q}$.

\appendix
\section*{Proof of Lemma~\ref{lem-lower4q}}\label{sec-app}

\begin{proof}
  Since $\mathbb{A}$ is the point set corresponding to a set of linearly dependent columns of $H(2,q)$ over $\F_2$, it intersects with each line of the line set $L(2,q)$ at an even number of points. We construct a graph $G_{\mathbb{A}}$ based on $\mathbb{A}$, with $\mathbb{A}$ as its vertex set and the adjacency relation defined as follows: $P_i$ is adjacent to $P_j$ if and only if  $P_i, P_j$ are collinear, where $P_i, P_j\in \mathbb{A}$.

Now we choose a ``maximal'' bipartite subgraph from $G_{\mathbb{A}}$. Let
  $$
  \mathcal{B} = \{\mathbb{A}_i \subseteq \mathbb{A}: \ \textrm{every two points of $\mathbb{A}_i$ are non-collinear and $| U(\mathbb{A}_i)|=q$} \} ,
  $$ and define
  \begin{eqnarray*}
    s & = & \max\{| \mathbb{A}_i|: \mathbb{A}_i \in \mathcal{B}\},\\
    \mathcal{C} & = & \{\mathbb{A}_i \in \mathcal{B}: | \mathbb{A}_i|=s\},\\
    t & = & \max \{| U(\mathbb{A}_i)\cap \mathbb{A}|: \mathbb{A}_i \in \mathcal{C}\},\\
    \mathcal{D} & = & \{\mathbb{A}_i \in \mathcal{C}: |U(\mathbb{A}_i)\cap \mathbb{A}|=t\}.
\end{eqnarray*}

Then let $\mathbb{A}_{\max}\in \mathcal{D}$ and $\mathbb{B}_{\max}=U(\mathbb{A}_{\max})\cap \mathbb{A}$. Notice that $|\mathbb{A}_{\max}| = s$ and $|\mathbb{B}_{\max}|=t$. By Remark~\ref{rmk-common}, we obtain a bipartite subgraph $(\mathbb{A}_{\max}, \mathbb{B}_{\max})$ of $G_{\mathbb{A}}$. It is easy to see that $2 \leq |\mathbb{A}_{\max}| \leq q$ and $0 \leq |\mathbb{B}_{\max}| \leq |\mathbb{A}_{\max}|$. Suppose that $\mathbb{A}_{\max}=\{P_1, \ldots, P_s \}$ and $\mathbb{B}_{\max}=\{Q_1, \ldots, Q_t \} $. We are going to find two more point subsets $\mathbb{C}_{\max}, \mathbb{D}_{\max}$ of $\mathbb{A}$ such that these two sets together with $\mathbb{A}_{\max}$ and $\mathbb{B}_{\max}$ are pairwise disjoint. Furthermore, the set $\mathbb{C}_{\max}$ has the property that each point of $\mathbb{C}_{\max}$ is collinear with some point of $\mathbb{A}_{\max}$. Similarly, each point of $\mathbb{D}_{\max}$ is collinear with some point of $\mathbb{B}_{\max}$.

By Lemma~\ref{lem-maxset}, there are $q+1$ lines passing through $P_{i}$ for every $1\leq i\leq s$. Hence besides the lines $\ell_{i,k} = \{P_i + x(Q_k - P_i):x\in \F_q\}$ with $1 \leq k \leq t$, there remain $q+1-t$ other lines passing through $P_{i}$, here also denoted by $\ell_{i,k}$ with $t+1 \leq k \leq q+1$. For each $t+1 \leq k \leq q+1$, the line $\ell_{i,k}$ contains at least one point $Q_{i,k}$ of $\mathbb{A}$ that is distinct from $P_i$. Let $\mathbb{C}_i = \{ Q_{i,t+1},
\ldots, Q_{i,q+1} \}$, then it is verified that $\mathbb{C}_i \cap \mathbb{B}_{\max} = \emptyset$ for every $ 1 \leq i \leq s$.  We claim that $\mathbb{C}_i \cap \mathbb{C}_j = \emptyset$ for all $1 \leq  i < j \leq s$. Indeed, if there is a certain $k$ with $t+1 \leq k \leq q+1$ such that $Q_{i,k} \in \mathbb{C}_i \cap \mathbb{C}_j$, then $Q_{i,k}\in U(P_i ) \cap U(P_j)$. It follows that $Q_{i,k}\in U(\mathbb{A}_{\max})\cap \mathbb{A} = \mathbb{B}_{\max}$ (see
Remark~\ref{rmk-common}). This is a contradiction to $Q_{i,k} \notin \mathbb{B}_{\max}$. Hence, we have $\mathbb{C}_i \cap \mathbb{C}_j = \emptyset$, for all $1 \leq i < j\leq s$, as claimed. Let $\mathbb{C}_{\max} = \bigcup\limits_{i=1}^{s} \mathbb{C}_i$, then $\mathbb{C}_{\max} \cap \mathbb{B}_{\max}=\emptyset$. Moreover, we claim that $\mathbb{C}_{\max} \cap \mathbb{A}_{\max} = \emptyset$. Suppose on the contrary that there exist points $P_{l}\in
\mathbb{A}_{\max}$ and $Q_{i,k}\in \mathbb{C}_{\max}$, such that $P_l = Q_{i,k}$. Clearly $l \neq i$. Since $P_l = Q_{i,k}$ and $P_i$ are on the same line $\ell_{i,k}$, this is a contradiction to the fact that every two points of $\mathbb{A}_{\max}$ are non-collinear. Therefore, we have $\mathbb{C}_{\max}\cap \mathbb{A}_{\max} = \emptyset$.

As discussed above, we just obtained the point set $\mathbb{C}_{\max}$ by the points that are collinear with some point of $\mathbb{A}_{\max}$. In a similar way, by the points that are collinear with some point of $\mathbb{B}_{\max}$, we obtain a point set $\mathbb{D}_{\max} = \bigcup\limits_{i=1}^{t} \mathbb{D}_i$, where $\mathbb{D}_i = \{P_{i,s+1}, \ldots, P_{i,q+1}\}$, $P_{i,j}\in \mathbb{A} \setminus \mathbb{A}_{\max}$ is collinear with the point $Q_i$ of
$\mathbb{B}_{\max}$ and every two points of $\mathbb{D}_i$ are non-collinear. With similar argument to what we used for $\mathbb{C}_{\max}$, we can show that $\mathbb{D}_{\max} \cap \mathbb{A}_{\max} = \mathbb{D}_{\max} \cap \mathbb{B}_{\max} = \emptyset$.

Now we show that $\mathbb{C}_{\max}\cap \mathbb{D}_{\max} = \emptyset$ by contradiction. Assume that there exists $P_{i,j}\in \mathbb{D}_{\max}$ and $Q_{k,l}\in \mathbb{C}_{\max}$ with $P_{i,j} = Q_{k,l}$. Then by construction of $\mathbb{C}_{\max}$ and $\mathbb{D}_{\max}$, $P_{i,j}$ is collinear with $Q_i$ and $Q_{k,l}$ is collinear with $P_{k}$. Since $\mathbb{B}_{\max} = U(\mathbb{A}_{\max}) \cap \mathbb{A}$, $Q_{i}$ and $P_k$ are collinear. Therefore,
\begin{align*}
  P_{i,j}&\sim \{P_{i,j}+x(Q_{i}-P_{i,j}):x\in \F_q\} \sim Q_{i}\\
  &\sim \{Q_{i}+x(P_{k}-Q_{i}): x \in \F_q\} \sim P_k\\
  &\sim \{P_{k}+x(Q_{k,l}-P_{k}): x \in \F_q\} \sim Q_{k,l} = P_{i,j}
\end{align*}
is a cycle of length 6, contradicting to Theorem~\ref{thm-girth}.

Summing up what we have discussed above, we know that $\mathbb{A}_{\max}, \mathbb{B}_{\max}, \mathbb{C}_{\max}$ and $\mathbb{D}_{\max}$ are pairwise disjoint subsets of $\mathbb{A}$. Hence, we have
$$
|\mathbb{A}| \geq |\mathbb{A}_{\max}| + |\mathbb{B}_{\max}| + |\mathbb{C}_{\max}| + |\mathbb{D}_{\max}| = (q+2)s+(q+2)t-2st.
$$

We now consider the following three cases.

\noindent Case 1: $t \geq 2$

Since $2\leq s \leq q,\ 2\leq t \leq s \leq q$, so $2\leq \sqrt{st} \leq q$, we have
$$
|\mathbb{A}| \geq (q+2)s+(q+2)t-2st \geq2(q+2)\sqrt{st}-2st = 4q + (\sqrt{st} - 2) (2q - 2 \sqrt{st}) \geq 4q.
$$

\noindent Case $2$: $t=1$

In this case, if $s \geq 3$, then $|\mathbb{A}| \geq 4q+2 > 4q$. Now assume that $s = 2$. We claim that $Q_{2,1}$ is collinear with at most one point in $\mathbb{B}_{\max} \cup \mathbb{C}_{\max} = \{Q_1, Q_{1,1}, \ldots, Q_{1,q}\}$. Suppose on the contrary that $Q_{2,1}$ is collinear with two distinct points of $\{Q_{1}, Q_{1,1}, \ldots, Q_{1,q}\}$, say $Q'$ and $Q''$. Since $\mathbb{B}_{\max} = U(\mathbb{A}_{\max}) \cap \mathbb{A}$ and by construction of $\mathbb{C}_{\max}$,
$P_1$ is collinear with every point in $\{Q_1, Q_{1,1}, \ldots, Q_{1,q}\}$. In particular, $P_{1}$ is collinear with $Q'$ and $Q''$. Hence we have found in $\mathbb{A}$ two subsets $\mathbb{A}'_{\max} = \{Q',Q''\}$ and $\mathbb{B}'_{\max} = \{P_{1}, Q_{21}\}$, such that $Q'$ and $Q''$ are non-collinear, $| U(\mathbb{A}'_{max})| = q$, and $\mathbb{B}'_{\max} \subseteq U(\mathbb{A}'_{\max}) \cap \mathbb{A}$, which contradicts to $s=2$ and $t=1$. Similarly, $Q_{2,1}$ is non-collinear with any point of $\mathbb{A}_{\max} \cup \mathbb{D}_{\max} = \{P_{1}, P_{2}, P_{1,3}, \ldots, P_{1,q+1}\}$ except for $P_2$. Hence, $Q_{2,1}$ is collinear with at most 2 points of $\mathbb{A}_{\max} \cup \mathbb{B}_{\max} \cup \mathbb{C}_{\max} \cup \mathbb{D}_{\max}$. On the other hand, by Lemma \ref{lem-maxset}, there are $q+1$ distinct lines passing through $Q_{2,1}$. Therefore, there are at least $q-1$ lines passing through $Q_{2,1}$ such that none of them intersect with $\mathbb{A}_{\max} \cup
\mathbb{B}_{\max} \cup \mathbb{C}_{\max} \cup \mathbb{D}_{\max} \setminus \{Q_{2,1}\}$. Denote these lines by $\ell_j$ with $1 \leq j \leq m$ and $m\geq q-1$. Every line $\ell_j$ contains at least one point of $\mathbb{A}$ distinct from $Q_{2,1}$, which we denote by $E_j$. Let $\mathbb{E}_{\max} = \{E_1, \ldots, E_{m}\}$. Clearly we have $\mathbb{A}_{\max} \cup \mathbb{B}_{\max} \cup \mathbb{C}_{\max} \cup \mathbb{D}_{\max} \cup \mathbb{E}_{\max} \subseteq \mathbb{A}$ and $|\mathbb{E}_{\max}| = m \geq q-1$. Thus, $|\mathbb{A}| \geq 3q+2+(q-1)=4q+1>4q$.

\noindent Case 3: $t=0$

In this case, it is easily seen that $s \geq 3$. Consequently, we only have two point sets $\mathbb{A}_{\max}$ and $\mathbb{C}_{\max}$, with  $|\mathbb{A}| \geq (q+2)s $. If $s \geq 4$, $|\mathbb{A}| \geq 4(q+2) > 4q$. Consider in the following that $s=3$. We claim that $Q_{2,1}$ is collinear with at most two points in $\{ Q_{1,1}, \ldots, Q_{1,q+1}\}$. If it is not the case, suppose that $Q_{2,1}$ is collinear with $Q_{1,i}, Q_{1,j}, Q_{1,k}$ with $1\leq i< j< k\leq q+1$.
Then we have $\{P_{1}, Q_{2,1}\}\subseteq U(\{Q_{1,i},Q_{1,j},Q_{1,k}\})$. By Lemma~\ref{lem-threecommon}, we have $|U(\{Q_{1,i},Q_{1,j},Q_{1,k}\})| = q$. Let $\mathbb{A}'_{\max} = \{Q_{1,i}, Q_{1,j},Q_{1,k}\}$ and $\mathbb{B}'_{\max} = \{P_1, Q_{2,1}\}$, then every two points of $\mathbb{A}'_{\max}$ are non-collinear, $| U(\mathbb{A}'_{max})| = q$, and $\mathbb{B}'_{\max} \subseteq U(\mathbb{A}'_{\max}) \cap \mathbb{A}$, contradicting to $s=3$ and $t=0$. Similarly, $Q_{2,1}$ is collinear
with at most $2$ points in $\{ Q_{3,1}, \ldots, Q_{3,q+1}\}$. Moreover, we show that $Q_{2,1}$ is non-collinear with $P_1$ or $P_3$. Indeed, if for instance $Q_{2,1}$ is collinear with $P_1$, then $Q_{2,1}\in U(\{P_1,P_2\})=U(\{P_{1},P_{2},P_{3}\})$ (see Remark~\ref{rmk-common}), which is a contradiction to $t=0$. Hence, $Q_{2,1}$ is collinear with at most $4$ points of $\mathbb{A}_{\max} \cup \mathbb{C}_{\max}$. Besides these lines, there are at least $q-3$ lines passing
through $Q_{2,1}$, which we denote by $\ell_j$ with $1\leq j\leq m$ and $m\geq q-3$. Every line $\ell_j$ contains at least one point $E_j$ of $\mathbb{A}$ that is distinct from $Q_{2,1}$. Let $\mathbb{E}_{\max}=\{E_1,\ldots, E_m\}$. It is clear that $\mathbb{A}_{\max} \cup \mathbb{C}_{\max} \cup \mathbb{E}_{\max} \subseteq \mathbb{A}$ and $|\mathbb{E}_{\max}| = m \geq q-3$. Therefore, $|\mathbb{A}| \geq (3q+6)+(q-3) = 4q+3 > 4q$.

The proof is thus completed.
\end{proof}





\end{document}